\documentclass[12pt]{article}

\usepackage[T1]{fontenc}
\usepackage[active]{srcltx}
\usepackage{lmodern}
\usepackage{amsmath,amstext,amssymb,graphicx,graphics,color}
\usepackage{tikz}
\usepackage{pgfplots}
\usepackage{theorem}
\usepackage{cite}               
\usepackage{hyperref}           
\usepackage{bbm}                
\usepackage{fancybox}           
\usepackage[section]{placeins}  
\usepackage{subcaption}
\theorembodyfont{\normalfont\slshape} 
\newtheorem{definition}{Definition}
\newtheorem{theorem}{Theorem}
\newtheorem{proposition}{Proposition}[section]
\newtheorem{corollary}[proposition]{Corollary}

\newtheorem{lemma}[proposition]{Lemma}
\theoremstyle{break} 

\newenvironment{proof}%
{{\par\noindent \bf Proof. \nobreak}}%
{\nobreak \removelastskip \nobreak \hfill $\Box$ \medbreak}
{{\par\noindent \bf Proof \nobreak}}%
{\nobreak \removelastskip \nobreak \hfill $\Box$ \medbreak}
{{\par\noindent \bf Proof lemma. \nobreak}}%
{\nobreak \removelastskip \nobreak \bf End proof lemma. \medbreak}
\newenvironment{remark}{\par \medskip \noindent {\bf Remark. }\nobreak}{\par \medskip}
\usepackage{fancyhdr}
 \fancypagestyle{plain}{
  \fancyhead{}
  
}

\def\paragraph#1{{\bf #1\ }}
\newcommand{\expo}{\mathrm{e}}

\newcommand{\dd}{\mathrm{d}}

\newcommand{\HH}{\mathrm{H}}
\newcommand{\overbar}[1]{\mkern 1.5mu\overline{\mkern-1.5mu#1\mkern-1.5mu}\mkern 1.5mu}
\usepackage[utf8]{inputenc}
\usepackage{unicode_character}

\title{Derivation of wealth distributions from biased exchange of money}

\author{Fei Cao\footnotemark[1] \and Sebastien Motsch\footnotemark[1]}

\begin{document}
\maketitle

\footnotetext[1]{Arizona State University - School of Mathematical and Statistical Sciences, 900 S Palm Walk, Tempe, AZ 85287-1804, USA}

\tableofcontents

\begin{abstract}
In the manuscript, we are interested in using kinetic theory to better understand the time evolution of wealth distribution and their large scale behavior such as the evolution of inequality (e.g. Gini index). We investigate three type of dynamics denoted unbiased, poor-biased and rich-biased exchange models. At the individual level, one agent is picked randomly based on its wealth and one of its dollar is redistributed among the population. Proving the so-called propagation of chaos, we identify the limit of each dynamics as the number of individual approaches infinity using both coupling techniques \cite{sznitman_topics_1991} and martingale-based approach \cite{merle_cutoff_2019}. Equipped with the limit equation, we identify and prove the convergence to specific equilibrium for both the unbiased and poor-biased dynamics. In the rich-biased dynamics however, we observe a more complex behavior where a dispersive wave emerges. Although the dispersive wave is vanishing in time, its also accumulates all the wealth leading to a Gini approaching $1$ (its maximum value). We characterize numerically the behavior of dispersive wave but further analytic investigation is needed to derive such dispersive wave directly from the dynamics.
\end{abstract}




\section{Introduction}

Econophysics is an emerging branch of statistical physics that apply concepts and techniques of traditional physics to economics and finance \cite{savoiu_econophysics:_2013,chatterjee_econophysics_2007,dragulescu_statistical_2000}. It has attracted considerable attention in recent years raising challenges on how various economical phenomena could be explained by universal laws in statistical physics, and we refer to \cite{chakraborti_econophysics_2011,chakraborti_econophysics_2011-1,pereira_econophysics_2017,kutner_econophysics_2019}  for a general review.

The primary motivation for study models arising from econophysics is at least two-fold: from the perspective of a policy maker, it is important to deal with the raise of income inequality \cite{dabla-norris_causes_2015,de_haan_finance_2017} in order to establish a more egalitarian society. From a mathematical point of view, we have to understand the fundamental mechanisms, such as money exchange resulting from individuals, which are usually agent-based models. Given an agent-based model, one is expected to identify the limit dynamics as the number of individuals tends to infinity and then its corresponding equilibrium when run the model for a sufficiently long time (if there is one), and this guiding approach is carried out in numerous works across different fields among literatures of applied mathematics, see for instance \cite{naldi_mathematical_2010,barbaro_phase_2014,carlen_kinetic_2013}.

Although we will only consider three distinct binary exchange models in the present work, other exchange rules can also be imposed and studied, leading to different models. To name a few, the so-called immediate exchange model introduced in \cite{heinsalu_kinetic_2014} assumes that pairs of agents are randomly and uniformly picked at each random time, and each of the agents transfer a random fraction of its money to the other agents, where these fractions are independent and uniformly distributed in $[0,1]$. The so-called uniform reshuffling model investigated in \cite{dragulescu_statistical_2000} and \cite{lanchier_rigorous_2018} suggests that the total amount of money of two randomly and uniformly picked agents possess before interaction is uniformly redistributed among the two agents after interaction. For models with saving propensity and with debts, we refer the readers to \cite{chakraborti_statistical_2000}, \cite{chatterjee_pareto_2004} and \cite{lanchier_rigorous_2018-1}.

\subsection{Unbiased/poor-biased/rich-biased dynamics}

In this work, we consider several dynamics for money exchange in a closed economical system, meaning that there are a fixed number of agents, denoted by $N$, with an (fixed) average number of dollar $m$. We denote by $S_i(t)$ the amount of dollars the agent $i$ has at time $t$. Since it is a closed economical system, we have:
\begin{equation}
  \label{eq:preserved_sum}
  S_1(t)+…+S_N(t) = \text{Constant} \qquad \text{for all } t\geq 0.
\end{equation}
As a first example of money exchange, we review the model proposed in \cite{dragulescu_statistical_2000}: at random time (exponential law), an agent $i$ is picked at random (uniformly) and if it has one dollar (i.e. $S_i\geq 1$) it will give it to another agent $j$ picked at random (uniformly). If $i$ does not have one dollar (i.e. $S_i= 0$), then nothing happens. From now on we will call this model as \textbf{unbiased exchange model} as all the agents are being picked with equal probability. We refer to this dynamics as follow:
\begin{equation}
  \label{unbiased_exchange}
  \textbf{unbiased:} \qquad   (S_i,S_j) \begin{tikzpicture} \draw [->,decorate,decoration={snake,amplitude=.4mm,segment length=2mm,post length=1mm}]
    (0,0) -- (.6,0); \node[above,red] at (0.3,0) {\small{$λ$}};\end{tikzpicture}  (S_i-1,S_j+1) \quad (\text{if } S_i\geq 1).
\end{equation}
In other words, any agents with at least one dollar gives to all of the others agents at a fixed rate. Later on, we will adjust the rate $λ$ (more exactly $λ\mathbbm{1}_{[1\!,+∞)}(S_i)$) by normalizing by $N$ in order to have the correct asymptotic as $N→+∞$ (the rate of one agent giving a dollar per unit time is of order $N$ otherwise).

Another possible dynamics is to pick the giver agent, i.e. agent $i$, with higher probability if the agent is rich, i.e. $S_i$ large. Thus, {\it poor} agent will have a lower frequency of being picked. From now on we will call this model as \textbf{poor-biased model} and it illustrates as follow:
\begin{equation}
  \label{poor_biased}
  \textbf{poor-biased:} \qquad   (S_i,S_j) \begin{tikzpicture} \draw [->,decorate,decoration={snake,amplitude=.4mm,segment length=2mm,post length=1mm}]
    (0,0) -- (.6,0); \node[above,red] at (0.3,0) {\small{$λS_i$}};\end{tikzpicture}  (S_i-1,S_j+1).
\end{equation}
Notice that since the rate of giving is $S_i$, an agent with no money, i.e. $S_i=0$, will never have to give. As for the unbiased dynamics \eqref{unbiased_exchange}, we will also adjust the rate, normalizing by $N$.

Our third dynamics that we would like to explore is the \textbf{rich-biased model}: we reverse the bias compared to the previous dynamics, rich agents are {\it less} likely to give:
\begin{equation}
  \label{rich_biased}
  \textbf{rich-biased:} \qquad   (S_i,S_j) \begin{tikzpicture} \draw [->,decorate,decoration={snake,amplitude=.4mm,segment length=2mm,post length=1mm}]
    (0,0) -- (.6,0); \node[above,red] at (0.3,0) {\small{$λ/S_i$}};\end{tikzpicture}  (S_i-1,S_j+1) \quad (\text{if } S_i\geq 1).
\end{equation}
As a consequence of this dynamics, rich agents will tend to become even richer compared to poor agents creating a feedback that could lead to singular behavior. The adjustment of the rate for this dynamics is more delicate since the sum of the rates $λ/S_i$ is no longer constant. In particular, we will see that a normalization of the rates to have a constant rate of giving a dollar per agent will lead to finite time blow-up of the dynamics in the limit $N→+∞$.

We illustrate the dynamics in figure \ref{fig:illustration_model}-left. The key question of interest is the exploration of the limiting money distribution among the agents as the total number of agents and the number of time steps become large. We illustrate numerically (see figure \ref{fig:illustration_3_dynamics_numerics}) the three previous dynamics using $N=500$ agents. In the unbiased dynamics (pink), the wealth distribution is (approximately) exponential with the proportion of agent decaying as wealth increases. On the contrary, the poor-biased dynamics (blue) has the bulk of  its distribution around $\$ 10$ (the average capital per agent). For the rich-biased dynamics (green), most of the agents are left with no money and few with large amounts (more than $\$ 30$). To visualize the temporal evolution of the three dynamics, we estimate the Gini index $G$ after each iteration in figure \ref{fig:illustration_model}-right:
\begin{equation}
  \label{eq:gini}
  G = \frac{1}{2μ}∑_{1\leq i,j \leq N} |S_i-S_j|,
\end{equation}
where $μ$ is the average wealth ($μ=\frac{1}{N}∑_{i=1}^N S_i$). The widely used inequality indicator Gini index $G$ measures the inequality in the wealth distribution and ranges from $0$ (no inequality) to $1$ (extreme inequality). Since all agents have the same amount of dollar initially ($S_i(t=0)=μ$), the Gini index starts at zero (i.e. $G(t=0)=0$). In the unbiased dynamics, the Gini index stabilizes around $.5$ (which corresponds to the Gini index of an exponential distribution). The Gini index is strongly reduced in the poor-biased dynamics ($G\approx .19$). On the contrary, the Gini index keeps increasing in the rich-biased dynamics and seems to approach $1$ (its maximum). We study in more details this phenoma in section \ref{sec:dispersive_wave}. We emphasize that the ``rich-get-richer'' phenomenon, numerically observed in the rich-biased dynamics in the present work, has also been reported in other models from econophysics, and we refer interested readers to \cite{boghosian_oligarchy_2017,boghosian_h_2015} and references therein.


\begin{figure}[p]
  \centering
  \includegraphics[width=.8\textwidth]{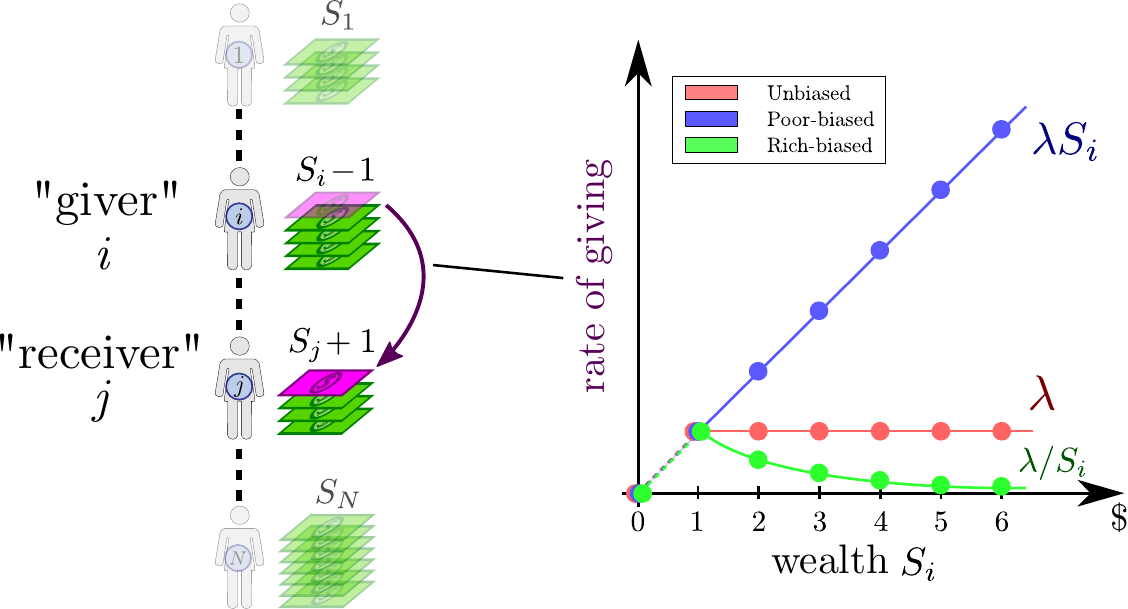}
  \caption{{\bf Left:} Illustration of the $3$ dynamics: at random time, one dollar is passed from a ``giver'' $i$ to a ``receiver'' $j$. {\bf Right:} The rate of picking the ``giver'' $i$ depends on the wealth $S_i$.}
  \label{fig:illustration_model}
\end{figure}

\begin{figure}[p]
  \centering
  \includegraphics[width=.47\textwidth]{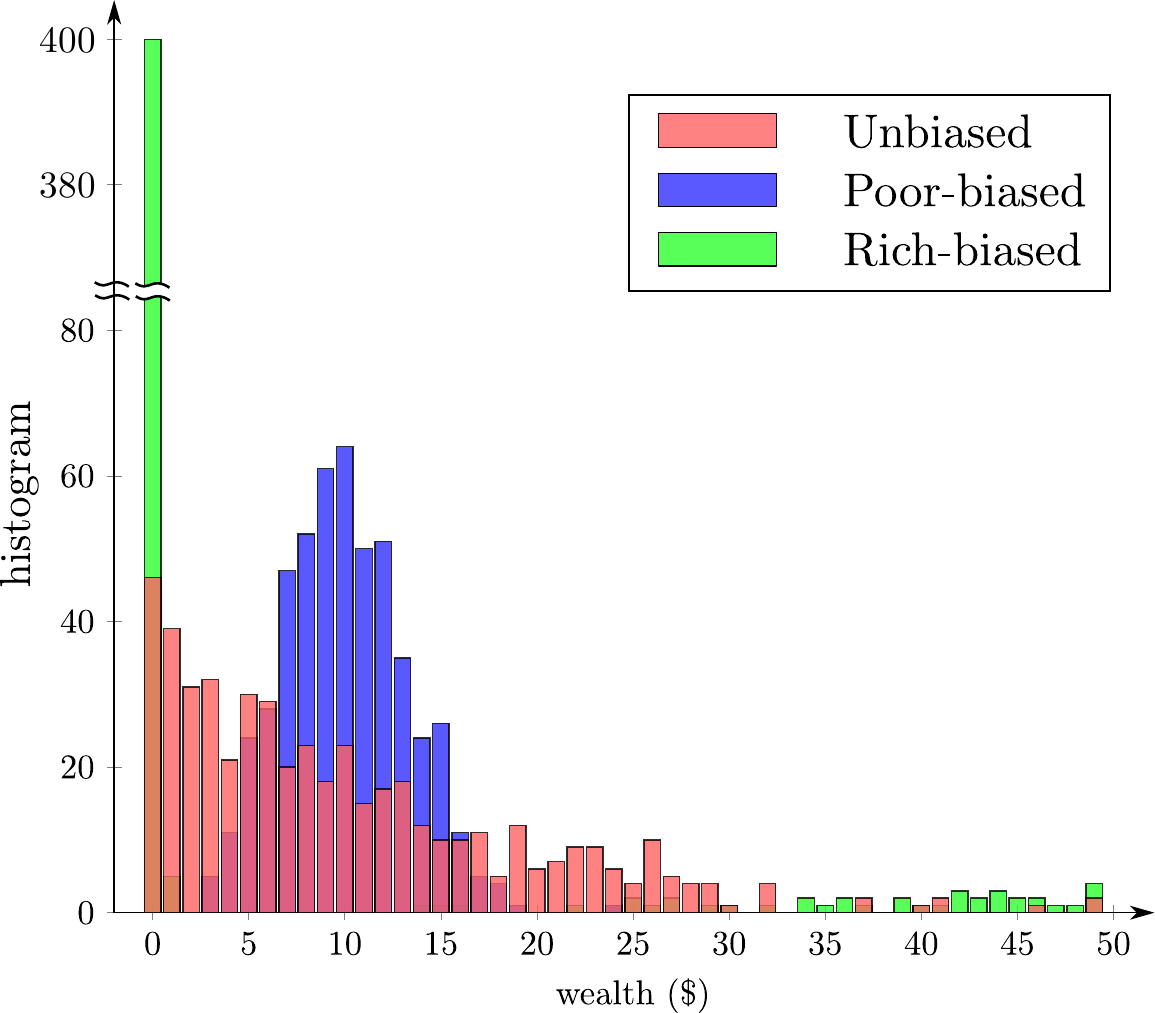}
  \includegraphics[width=.42\textwidth]{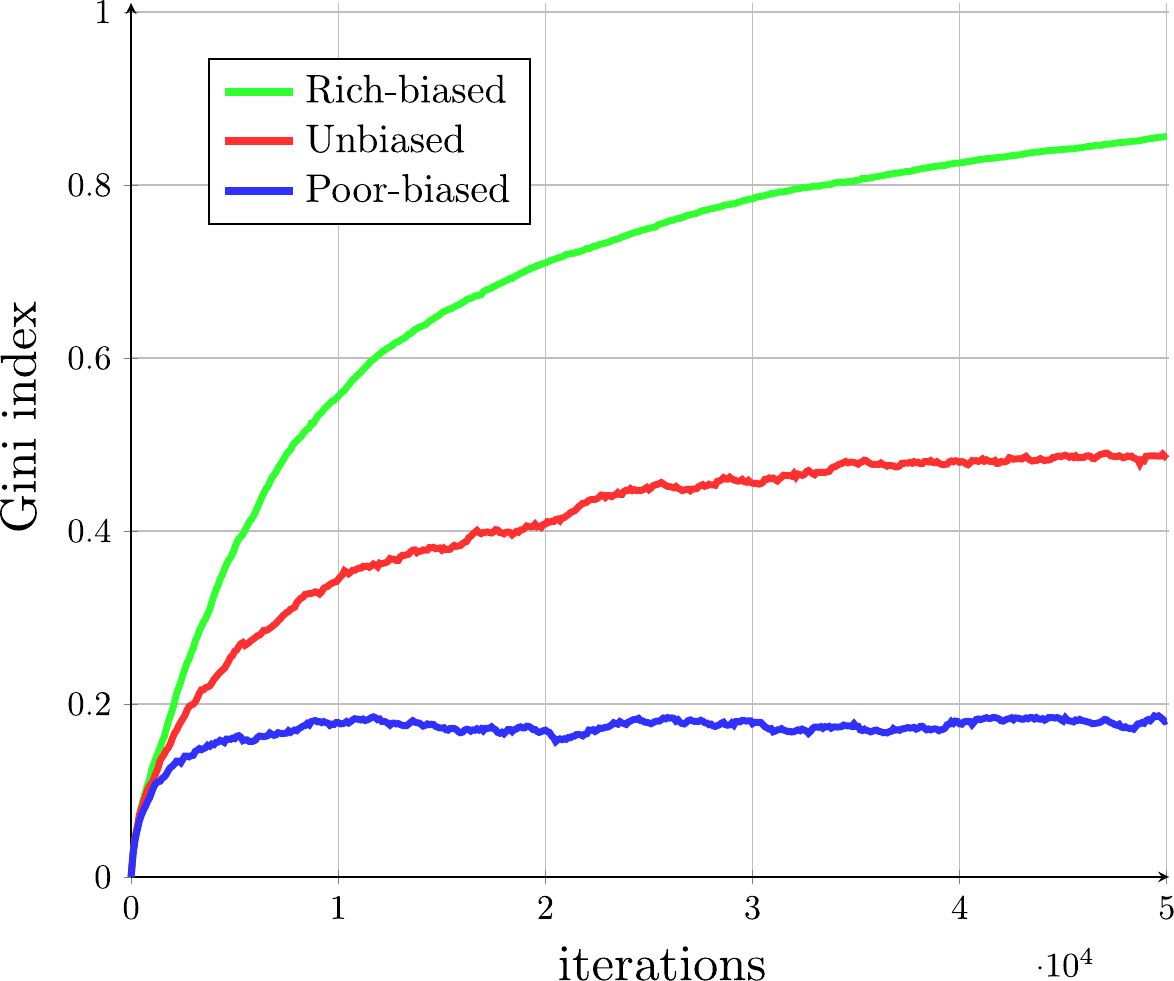}
  \caption{{\bf Left:} Distribution of wealth for the three dynamics after $50,000$ steps. The distribution decays for the unbiased dynamics (pink) i.e. poor agents are more frequent than rich agents, whereas in the poor-biased dynamics, the distribution (blue) is centered at the average $\$ 10$. For the rich-biased dynamics, almost all agents have zero dollars except a few with a large amount (more than $\$ 30$). {\bf Right:} evolution of the Gini index \eqref{eq:gini} for the three dynamics. The Gini index is lower for the poor-biased dynamics (less inequality) whereas it is approaching $1$ for the rich-biased dynamics.}
  \label{fig:illustration_3_dynamics_numerics}
\end{figure}

\subsection{Asymptotic dynamics: $N→+∞$ and $t→+∞$}

One of the main difficulty in any rigorous mathematical treatment lies in the general fact that models in econophysics typically consist of a large number of interacting (coupled) economic agents. Fortunately the framework of kinetic theories allows simplification of the mathematical analysis of certain such models under some appropriate limit processes.~~
For the unbiased model \eqref{unbiased_exchange} and the poor-biased model \eqref{poor_biased}, instead of taking the large time limit and then the large population limit as in \cite{lanchier_rigorous_2017}, we first take the large population limit to achieve a transition from the large stochastic system of interacting agents to a deterministic system of ordinary differential equations by proving the so-called propagation of chaos \cite{sznitman_topics_1991,merle_cutoff_2019,meleard_propagation_1987,oelschlager_martingale_1984} through a well-designed coupling technique, see figure \ref{scheme} for a illustration of these strategies. After that, analysis of the deterministic description is then built on its (discrete) Fokker-Planck formulation and we investigate the convergence toward an equilibrium distribution by employing entropy methods. \cite{arnold_convex_2001,matthes_entropy_2007,jungel_entropy_2016}. For the rich-biased model, we prove the propagation of chaos by virtue of a novel martingale-based technique introduced in \cite{merle_cutoff_2019}, and we report some interesting numerical behavior of the associated ODE system. We illustrate the various (limiting) ODE systems obtained in the present work in figure \ref{ode_summary}.

\begin{figure}[!htb]
\centering
\includegraphics[scale=0.8]{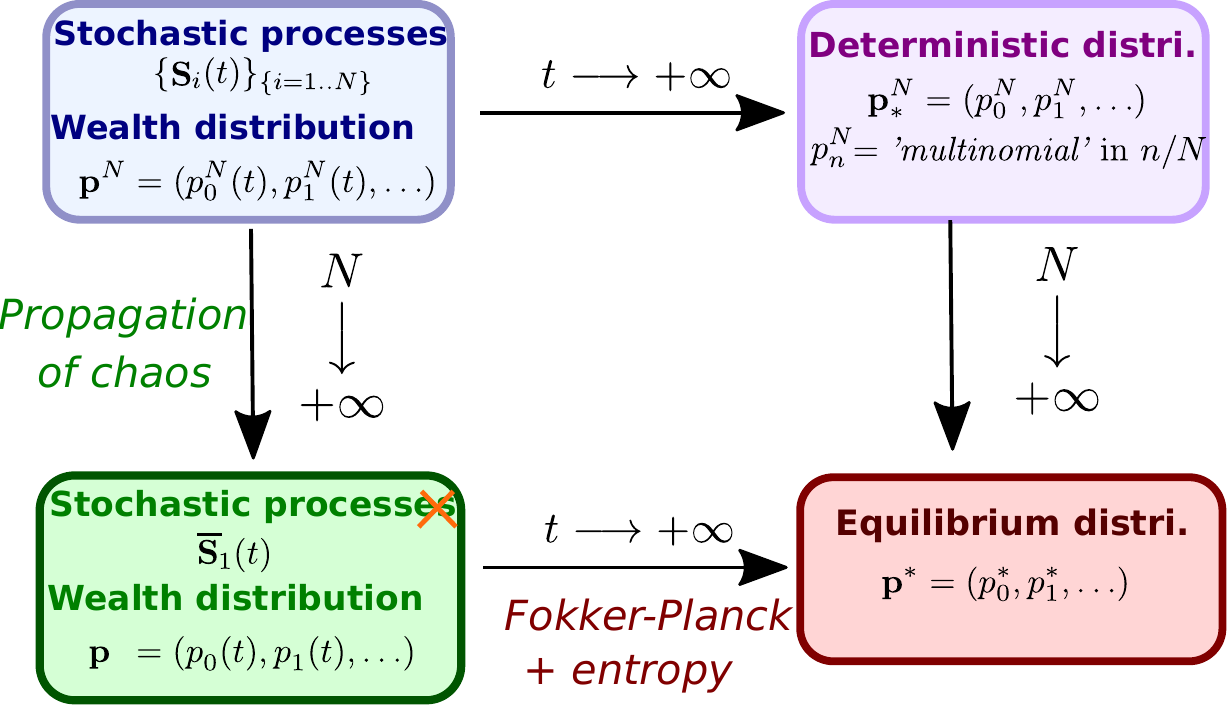}
\caption{Schematic illustration of the strategy of proof: The approach of sending $t\to \infty$ first and then taking $N \to \infty$ is carried out in \cite{lanchier_rigorous_2017} (see also \cite{lanchier_rigorous_2018,lanchier_rigorous_2018-1} for usage of this approach applied for a variety of models in econophysics). Our strategy is to perform the limit $N \to \infty$ before investigating the time asymptotic $t\to \infty$.}
\label{scheme} 
\end{figure}

\begin{figure}[!htb]
\centering
\includegraphics[scale=1.0]{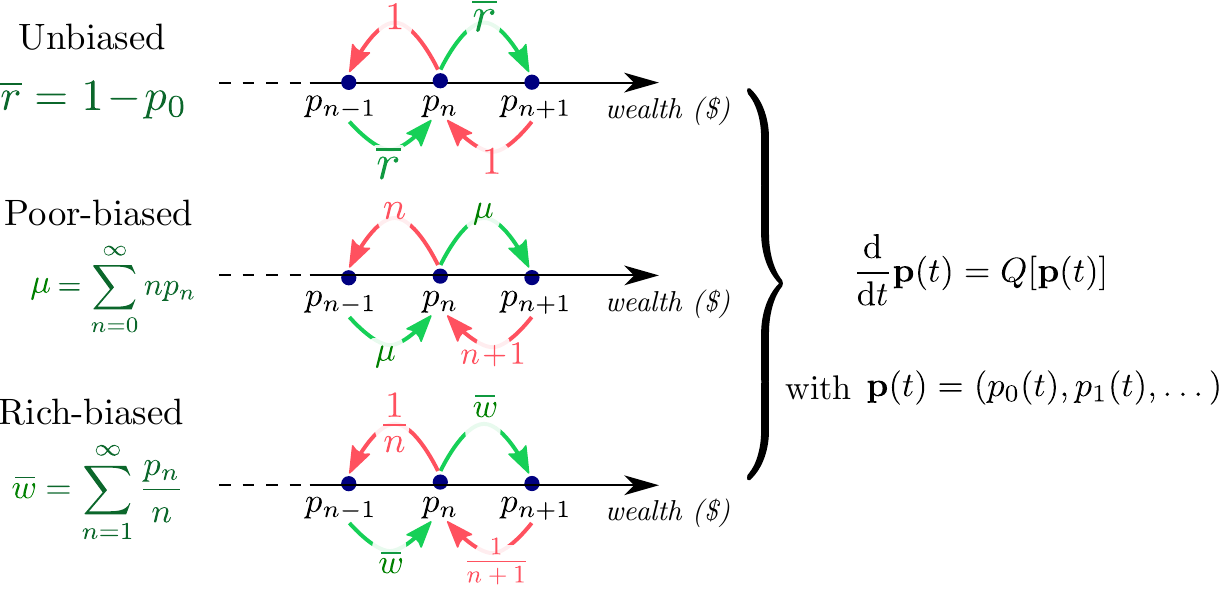}
\caption{Summary of the limit ODE systems obtained in this manuscript. The exact form of the operator $Q$ will be model-dependent.}
\label{ode_summary} 
\end{figure}


For the poor-biased model, we present an explicit rate of convergence of its associated system of ordinary differential equations toward its equilibrium via the Bakry-Emery approach \cite{bakry_diffusions_1985}. Then, we resort to numerical simulation in the determination of the sharp rate of convergence and a heuristic argument is used in support of our numerical observation.

This paper is organized as follows: in section \ref{sec:propagation_of_chaos}, we briefly review different approaches to tackle the propagation of chaos. Section \ref{sec:unbiased_exchange} is devoted to the investigation of the unbiased exchange model, where the rigorous large population limit $N \to \infty$ is carried out via a coupling argument and the limiting system of ODEs is studied in detail. We perform the analysis, for the poor-biased model in section \ref{sec:poor_biased_exchange} and for the rich-biased model in section \ref{sec:rich_biased_exchange}, in a parallel fashion that resembles section \ref{sec:unbiased_exchange}. A subsection is dedicated in \ref{sec:dispersive_wave} to the emergence of a dispersive traveling wave in the rich-biased dynamics. Finally, a conclusion is drawn in section \ref{sec:conclusion}.

\section{Review propagation of chaos}
\label{sec:propagation_of_chaos}
\setcounter{equation}{0}

\subsection{Definition}

We propose to review the method used to prove the so-called propagation of chaos. But first we need to carefully define what propagation of chaos means. With this aim, we consider a (stochastic) $N-$particle system denoted $(S_1,…,S_N)$ where particles are indistinguishable. In other words, the particle system is invariant by permutation, i.e. for any test function $φ$ and permutation $σ∈\mathcal{S}_N$:
\begin{displaymath}
  \mathbb{E}[φ(S_1,…,S_N)] = \mathbb{E}[φ(S_{σ(1)},…,S_{σ(N)})].
\end{displaymath}
In particular, all the single processes $S_i$ for $i=1,\ldots,N$ have the same law (but they are in general not independent). Denote by ${\bf p}^{(N)}(s_1,…,s_N)$ the density distribution of the $N-$process
and let ${\bf p}^{(N)}_k$  be the marginal density, i.e. the law of the process $(S_1,...,S_k)$ (for $1\leq k \leq N$):
\begin{displaymath}
  {\bf p}^{(N)}_k(s_1,…,s_k) = ∫_{s_{k+1},…,s_N} {\bf p}^{(N)}(s_1,…,s_N)\,\dd s_{k+1}…\dd s_N.
\end{displaymath}
Consider now a {\it limit} stochastic process $(\overline{S}_1,…,\overline{S}_k)$ where $\{\overline{S}_i\}_{i=1,\ldots,k}$ are independent and identically distributed. Denote by ${\bf p}_1$ the law of a single process, thus by independence assumption the law of {\it all} the processes is given by:
\begin{displaymath}
  {\bf p}_k(s_1,…,s_k) = \prod_{i=1}^k {\bf p}_1(s_i).
\end{displaymath}

\begin{definition}
  We say that the stochastic process $(S_1,…,S_N)$ satisfies the propagation of chaos if for any fixed $k$:
  \begin{equation}
    \label{eq:def_propa_pdf}
    {\bf p}^{(N)}_k \;\; \stackrel{N → +∞}{\rightharpoonup}\;\; {\bf p}_k
  \end{equation}
  which is equivalent to have for any test function $φ$:
  \begin{equation}
    \label{eq:process}
    \mathbb{E}[φ(S_1,…,S_k)] \stackrel{N → +∞}{⟶} \mathbb{E}[φ(\overline{S}_1,…,\overline{S}_k)].
  \end{equation}
\end{definition}

\begin{remark}
For binary collision models \cite{carlen_kinetic_2013-1,carlen_kinetic_2013}, proving propagation of chaos is equivalent to show that ${\bf p}^{(N)}_2(s_1,s_2) \approx {\bf p}^{(N)}_1(s_1)\,{\bf p}^{(N)}_1(s_2)$, i.e. collisions come from two independent particles.
\end{remark}

\subsection{Coupling method}

The {\it coupling} method \cite{sznitman_topics_1991} consists in generating the two processes $(S_1,…,S_N)$ and $(\overline{S}_1,…,\overline{S}_k)$ {\it simultaneously} in such a way that:
\begin{itemize}
\item[i)] $(S_1,…,S_k)$ and $(\overline{S}_1,…,\overline{S}_k)$ satisfy their respective law,
\item[ii)] $S_i$ and $\overline{S}_i$ are {\it closed} for all $1\leq i\leq k$.
\end{itemize}
The main difficulty is that $\{\overline{S}_i\}_{i=1,\ldots,k}$ are independent but $\{S_i\}_{i=1,\ldots,N}$ are \underline{not}, thus the two processes cannot be {\it too} closed. In practice, we expect to find a bound of the form:
\begin{equation}
  \label{eq:exp_error}
  \mathbb{E}[|S_i-\overline{S}_i|] \leq \frac{C}{\sqrt{N}} \stackrel{N → +∞}{⟶} 0 \quad , \quad \text{for all } 1 \leq i \leq k.
\end{equation}
Such result is sufficient\footnote{using as a test function $φ(s_1,…,s_k)=φ_1(s_1)…φ_k(s_k)$} to prove \eqref{eq:process} and therefore one deduces propagation of chaos.

In a more abstract point of view, the inequality \eqref{eq:exp_error} gives an upper bound for the Wasserstein distance between ${\bf p}_k^{(N)}$ and the limit density ${\bf p}_k$. Since convergence in Wasserstein distance is equivalent to weak-* convergence for measures, we can conclude about the propagation of chaos \eqref{eq:def_propa_pdf}.

\subsection{Empirical distribution - tightness of measure}
\label{sec:empirical_measure}

Another approach to prove propagation of chaos is to study the so-called empirical measure:
\begin{equation}
  \label{eq:empirical}
  {\bf p}_{emp}^{(N)}(s) = \frac{1}{N} ∑_{i=1}^N δ_{S_i}(s)
\end{equation}
where $δ$ is the Delta distribution, i.e. for a smooth test function $φ(s)$ the duality bracket is defined as:
\begin{equation}
  \label{eq:bracket}
  ⟨{\bf p}_{emp}^{(N)},φ⟩ = \frac{1}{N}∑_{i=1}^N φ(S_i).
\end{equation}
Notice that ${\bf p}_{emp}^{(N)}$ is a distribution of a single variable, thus the domain of ${\bf p}_{emp}^{(N)}$ remains the same as $N$ increases which simplifies its study. However, ${\bf p}_{emp}^{(N)}$  is also a {\it stochastic} measure, i.e. ${\bf p}_{emp}^{(N)}$ is a random variable on the space of measures \cite{billingsley_convergence_2013}. The link between propagation of chaos and empirical distribution relies on the following lemma.

\begin{lemma}
  \label{lem:tightness}
  The stochastic process $(S_1,\ldots,S_N)$ satisfies the propagation of chaos \eqref{eq:def_propa_pdf} if and only if:
  \begin{equation}
    \label{eq:cv_rho_emp}
    {\bf p}_{emp}^{(N)} \stackrel{N → +∞}{\rightharpoonup} {\bf p}_1,
  \end{equation}
  i.e. for any test function $φ$ the random variable $⟨{\bf p}_{emp}^{(N)},φ⟩$ converges in law to the constant value $\mathbb{E}[φ(\overline{S}_1)]$.
\end{lemma}
The proof can be found in \cite{sznitman_topics_1991} but for completeness we write our own in appendix \ref{lemma:proof_tightness}.


\section{Unbiased exchange model}
\label{sec:unbiased_exchange}
\setcounter{equation}{0}

\subsection{Definition and limit equation}

We consider first the unbiased model that is briefly mentioned in the introduction above. For the three models investigated in this work, we consider a (closed) economic market consisting of $N$ agents with $\mu$ dollars per agents for some (fixed) $\mu \in \mathbb N_+$, i.e. there are a total of $\mu N$ dollars.  We denote by $S_i(t)$ the amount of dollars that agent $i$ has (i.e. $S_i(t) \in \{0,\ldots,\mu N\}$ and $\sum_{i=1}^N S_i(t) = \mu N$ for any $t\geq 0$).

\begin{definition}[\textbf{Unbiased Exchange Model}]
  \label{def:unbiased_exchange_model}
  The dynamics consist in choosing with uniform probability a ``giver'' $i$ and a ``receiver'' $j$. If the receiver $i$  has at least one dollar (i.e. $S_i\geq 1$), then it gives one dollar to the receiver $j$. This exchange occurs according to a Poisson process with frequency $\lambda/N >0$.
\end{definition}

The unbiased exchange model can be written as a stochastic differential equation \cite{privault_stochastic_2013,shreve_stochastic_2004}. Introducing $\{\mathrm{N}_t^{(i,j)}\}_{1\leq i,j\leq N}$ independent Poisson processes with constant intensity $\frac{\lambda}{N}$, the evolution of each $S_i$ is given by:
\begin{equation}
  \label{UEMoriginal}
  \dd S_i(t)= -\sum \limits^{N}_{j=1} \underbrace{\mathbbm{1}_{[1,\infty)}\big(S_i(t-)\big) \dd \mathrm{N}^{(i,j)}_{t}}_{\text{``$i$ gives to $j$''}} + \sum \limits^{N}_{j=1} \underbrace{\mathbbm{1}_{[1,\infty)}\big(S_j(t-)\big) \dd \mathrm{N}^{(j,i)}_{t}}_{\text{``$j$ gives to $i$''}}.
\end{equation}
To gain some insight of the dynamics, we focus on $i=1$ and introduce some notations:
\begin{displaymath}
  \mathrm{\bf N}^1_t = \sum_{j=1}^N \mathrm{N}^{(1,j)}_t,\quad \mathrm{\bf M}^1_t = \sum_{j=1}^N \mathrm{N}^{(j,1)}_t.
\end{displaymath}
The two Poisson processes $\mathrm{\bf N}^1_t$ and $\mathrm{\bf M}^1_t$ are of intensity $λ$. The evolution of $S_1(t)$ can be written as:
\begin{equation}
  \label{UEMoriginal_rewrite}
  \dd S_1(t) = -\mathbbm{1}_{[1,\infty)}\big(S_1(t-)\big) \dd \mathrm{\bf N}^1_t + Y(t-) \dd \mathrm{\bf M}^1_t,
\end{equation}
with $Y(t)$ Bernoulli distribution with parameter $r(t)$ (i.e. $Y(t)\sim \mathcal{B}(r(t))$) representing the proportion of ``rich'' people:
\begin{equation}
  \label{eq:def_r1}
  r(t)=\frac{1}{N}\sum_{j=1}^N \mathbbm{1}_{[1,\infty)}\big(S_j(t)\big).
\end{equation}
Thus, the dynamics of $S_1$ can be seen as a compound Poisson process.

Motivated by \eqref{UEMoriginal_rewrite}, we give the following definition of the limiting dynamics of $S_1(t)$ as $N \rightarrow \infty$ from the process point of view.

\begin{definition}[\textbf{Asymptotic Unbiased Exchange Model}]
  We define $\overbar{S}_1(t)$ to be the (nonlinear) compound Poisson process satisfying the following SDE:
  \begin{equation}
    \label{UEMlimit_S1}
    \dd \overbar{S}_1(t) = -\mathbbm{1}_{[1,\infty)}\big(\overbar{S}_1(t-)\big) \dd \overbar{\mathrm{\bf N}}^1_t + \overbar{Y}(t-) \dd \overbar{\mathrm{\bf M}}^1_t,
  \end{equation}
  in which $\overbar{\mathrm{\bf N}}^1_t$ and $\overbar{\mathrm{\bf M}}^1_t$ are independent Poisson processes with intensity $\lambda$, and $\overbar{Y}(t) \sim \mathcal B(\overbar{r}(t))$ independent Bernoulli variable with parameter
  \begin{equation}
    \label{eq:def_r_bar}
    \overbar{r}(t)\;:=\; \mathbb P\big(\overbar{S}_1(t) > 0\big) \;\;=\;\; 1 - \mathbb P\big(\overbar{S}_1(t)=0\big).
  \end{equation}
\end{definition}
We denote by ${\bf p}(t)=\big(p_0(t),p_1(t),\ldots\big)$  the law of the process $\overbar{S}_1(t)$, i.e. $p_n(t) = \mathbb P\big(\overbar{S}_1(t) = n)$. Its time evolution is given by:
\begin{equation}
  \label{eq:law_limit_unbais}
  \frac{\dd}{\dd t} {\bf p}(t) = λ \,Q_{unbias}[{\bf p}(t)]
\end{equation}
with:
\begin{equation}
  \label{eq:Q_unbias}
  Q_{unbias}[{\bf p}]_n:= \left\{
    \begin{array}{ll}
      p_1-\overbar{r}\,p_0 & \quad \text{if } n=0 \\
      p_{n+1}+\overbar{r}\,p_{n-1}- (1+\overbar{r})p_n & \quad \text{for } n \geq 1
    \end{array}
  \right.
\end{equation}
and $\overbar{r}=1-p_0$.

\subsection{Coupling for the unbiased exchange model}

We now provide the coupling strategy to link the $N-$particle system $(S_1,…,S_N)$ with the limit dynamics $(\overline{S}_1,…,\overline{S}_k)$. In \cite{sznitman_topics_1991}, the core of the method is to use the same ``noise'' in both the $N-$particle system and the limit system. Unfortunately, it is not possible in our settings: the clocks $\mathrm{N}_t^{(i,j)}$ cannot be used ``as it'' since they would correlate the jump of $\overline{S}_i$ with the jump of $\overline{S}_j$ which is not acceptable. Indeed, if $\overline{S}_i(t)$ and $\overline{S}_j(t)$ are independent, they cannot jump at (exactly) the same time.

For this reason, we have to introduce an intermediate dynamics, denoted by $\{\widehat{S}_i\}_{i \geq 1}$, which employs exactly the same ``clocks'' as our original dynamics \eqref{UEMoriginal}, but the property of being rich or poor is decoupled.

\begin{definition}[\textbf{Intermediate model}]
  We define for $\{\widehat{S}_i\}_{1\leq i\leq N}$ to be a collection of identically distributed (nonlinear) compound Poisson processes satisfying the following SDEs for each $1\leq i\leq N$:
  \begin{eqnarray}
    \label{eq:intermediate}
    \dd \widehat{S}_i(t) &=& -\sum_{j=1,j\neq i}^{N} \mathbbm{1}_{[1,\infty)}\big(\widehat{S}_i(t-)\big) \dd \mathrm{N}^{(i,j)}_t + \sum_{j=1,j\neq i}^{N} \overbar{Y}(t-) \dd \mathrm{N}^{(j,i)}_t \\
                         && \quad -\mathbbm{1}_{[1,\infty)}\big(\widehat{S}_i(t-)\big) \dd \overbar{\mathrm{N}}^{(i,i)}_t +  \overbar{Y}(t-) \dd \overbar{\mathrm{M}}^{(i,i)}_t     \label{eq:extra_clock}
  \end{eqnarray}
  in which $\overbar{Y}(t) \sim \mathcal B(\overbar{r}(t))$, the Poisson clocks $\mathrm{N}^{(i,j)}_t$ ($1\leq i \neq j\leq N$) are the same as those used in \eqref{UEMoriginal}, the two extra clocks $\overbar{\mathrm{N}}^{(i,i)}_t$ and $\overbar{\mathrm{M}}^{(i,i)}_t$ are independent with rate $λ/N$.
\end{definition}
We do not use the ``self-giving'' clocks $\mathrm{N}^{(i,i)}_t$ since we want to decouple the receiving and giving dynamics.

\begin{figure}[!htb]
  \centering
  \includegraphics[width=.7\textwidth]{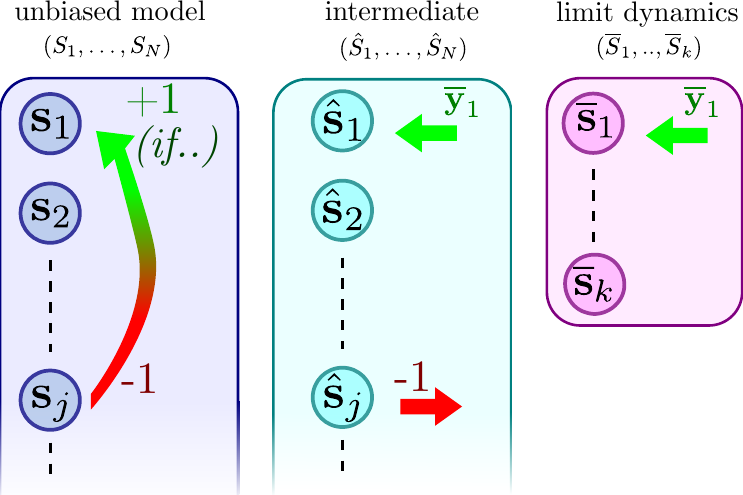}
  \caption{Schematic illustration of the coupling strategy. We use an intermediate process $(\widehat{S}_1,…,\widehat{S}_N)$ to {\it decouple} the ``give'' and ``receive'' parts of the dynamics.}
  \label{coupling} 
\end{figure}

An schematic illustration of the above coupling technique is shown in Fig \ref{coupling} below.
We first have to control the difference between the process $(S_1,…,S_N)$ and the intermediate dynamics $(\widehat{S}_1,…,\widehat{S}_N)$. The key idea is based on the following simple yet effective lemma that allows to create optimal coupling between two flipping coins \cite{den_hollander_probability_2012}.
\begin{lemma}\label{lem}
For any $p,q \in (0,1)$, there exist $X \sim \mathcal B(p)$ and $Y \sim \mathcal B(q)$ such that $\mathbb P(X \neq Y) = |p-q|$.
\end{lemma}
\begin{proof}
  Let $U\sim \mathcal{U}[0,1]$ a uniform random variable. Define the Bernoulli random variables as $X:=\mathbbm{1}_{[0,p)}(U)$ and $Y:=\mathbbm{1}_{[0,q)}(U)$. It is straightforward to show that $X\sim \mathcal B(p)$, $Y \sim \mathcal B(q)$ and  $\mathbb P(X \neq Y) = |p-q|$.
\end{proof}
More generally, if $N_t$ and $M_t$ are two inhomogeneous Poisson processes with rate $\lambda(t)$ and $\mu(t)$, respectively, then there exists a coupling such that
\begin{displaymath}
  \dd \mathbbm{E}[|N_t-M_t|] \leq |\lambda(t)-\mu(t)|\dd t.
\end{displaymath}
This leads to the following proposition.
\begin{proposition}
  \label{ppo:unbiased_S_hat_S}
  Let $\big(S_1,\ldots,S_N\big)$ and $\big(\widehat{S}_1,\ldots,\widehat{S}_N\big)$ be solution to \eqref{UEMoriginal} and \eqref{eq:intermediate} respectively, with the same initial condition. Then for any $1\leq i\leq N$, we have
  \begin{equation}
    \label{eq:aux_limit}
    \dd \mathbbm{E}[|S_i(t)-\widehat{S}_i(t)|]\; \leq \; λ\,\mathbbm{E}[|r(t)-\overbar{r}(t)|]\,\dd t \; +\; λ \frac{2}{N}\,\dd t,
  \end{equation}
  where $r(t)=\frac{1}{N} \sum_{j=1}^N \mathbbm{1}_{[1,\infty)}\big(S_j(t)\big)$ and $\overbar{r}(t)$ given by \eqref{eq:def_r_bar}.
\end{proposition}
\begin{proof} The processes $\widehat{S}_i(t)$ and $S_i(t)$ ``share'' the same clocks $\mathrm{N}_t^{(i,j)}$ and $\mathrm{N}_t^{(j,i)}$ for $j \neq i$. Denote the 'rich or not' random Bernoulli random variables:
  \begin{equation}
    \label{eq:R_i}
    R_i(t)=\mathbbm{1}_{[1,\infty)}\big(S_i(t)\big) \quad \text{ and } \quad \widehat{R}_i(t)=\mathbbm{1}_{[1,\infty)}\big(\widehat{S}_i(t)\big).
  \end{equation}
  Once a clock $\mathrm{N}_t^{(i,j)}$ rings, the processes become:
  \begin{equation}
    \label{eq:update_algo}
    \begin{array}{ccc}
      (S_i,S_j) &\begin{tikzpicture} \draw [->,decorate,decoration={snake,amplitude=.4mm,segment length=2mm,post length=1mm}]
        (0,0) -- (.6,0);\end{tikzpicture}& (S_i-R_i,S_j+R_i), \\
      (\widehat{S}_i,\widehat{S}_j) &\begin{tikzpicture} \draw [->,decorate,decoration={snake,amplitude=.4mm,segment length=2mm,post length=1mm}]
        (0,0) -- (.6,0);\end{tikzpicture}& (\widehat{S}_i-\widehat{R}_i,\widehat{S}_j+\overbar{Y}).
    \end{array}
  \end{equation}
  Notice that the difference $|S_i-\widehat{S}_i|$ can only decay after the jump from the clock $\mathrm{N}_t^{(i,j)}$ (the 'give' dynamics reduce the difference). However, the 'receive' dynamics from the clock $\mathrm{N}_t^{(j,i)}$ could increase the difference $|S_j-\widehat{S}_j|$ if $\widehat{R}_i\neq \overbar{Y}$. More precisely, we find:
  \begin{equation}
    \label{eq:almost_ppo1}
    \dd 𝔼[|S_i(t)-\widehat{S}_i(t)|] \leq 0 +\sum_{j=1,j\neq i}^N 𝔼[|R_j(t-)- \overbar{Y}(t-)|]\frac{λ}{N}\,\dd t    \;+\;  \frac{2λ}{N} \,\dd t
  \end{equation}
  where the extra $\frac{2λ}{N}\,\dd t$ is due to the extra clocks $\overbar{\mathrm{N}}^{(i,i)}_t$ and $\overbar{\mathrm{M}}^{(i,i)}_t$ in \eqref{eq:extra_clock}.

  Now we  have to couple the Bernoulli process $\overbar{Y}(t-)$ with $R_j(t-)$ in a convenient way to make the difference as small as possible. Here is the strategy:
  \begin{itemize}
  \item Step 1: generate a master Poisson clock $\mathrm{\bf N}_t$ with intensity $λN$ which gives a collection of jumping times.
  \item Step 2: to select which clock $\mathrm{N}_t^{(i,j)}$ rings, calculate the proportions of ``rich people'' for the $N-$particle system and for the limit dynamics:
    \begin{equation}
      \label{eq:rich_particle_vs_limit}
      r(t-)=\frac{1}{N}\sum_{j=1}^N \mathbbm{1}_{[1,\infty)}\big(S_j(t-)\big) \quad,\quad \overbar{r}(t-) = 1-p_0(t-).
    \end{equation}
  \item Step 3: let $U\sim \mathcal{U}([0,1])$ a uniform random variable.
    \begin{itemize}
    \item if $U < r(t-)$, pick an index $i$ uniformly among the rich people (i.e. $i$ such that $S_i(t-)>0$), otherwise we pick $i$ uniformly among the poor people (i.e. $i$ such that $S_i(t-)=0$). Pick index $j$ uniformly among $\{1,2,\ldots,N\}$.
    \item if $U<\overbar{r}(t-)$, let $\overbar{Y}(t-)=1$, otherwise $\overbar{Y}(t-)=0$ (i.e. $\overbar{Y}(t-) =\mathbbm{1}_{[0,\overbar{r}(t-)]}(U)$).
    \end{itemize}
  \item Step 4: if $i\neq j$, update using \eqref{eq:update_algo}
  \end{itemize}
  Thanks to our coupling, the 'receiving' dynamics of  $S_i$ and $\widehat{S}_i$ will differ with probability $|r-\overbar{r}|$:
  \begin{equation}
    \label{eq:r_R}
    𝔼[|R_j(t-)- \overbar{Y}(t-)|] = \mathbbm{P}\big(R_j(t-) \neq \overbar{Y}(t-)\big)=\mathbbm{E}\left[|r-\overbar{r}|\right].
  \end{equation}
  Plug in the expression in \eqref{eq:almost_ppo1} concludes the proof.

\end{proof}

\begin{remark}
  The update formula \eqref{eq:update_algo} for $(\widehat{S}_i,\widehat{S}_j)$ highlights that the 'give' and 'receive' dynamics are now independent in the auxiliary dynamics (i.e. $\widehat{R}_i$ and $\overbar{Y}$ are independent). In contrast, we use the same process $R_i$ to update $S_i$ and $S_j$.
\end{remark}

Now we turn our attention to the coupling between the auxiliary dynamics $(\widehat{S}_1,\ldots,\widehat{S}_N)$ and the limit dynamics  $(\overbar{S}_1,\ldots,\overbar{S}_k)$ for a \textbf{fixed} $k$ (while $N \rightarrow \infty$). The idea is to remove the clocks $\mathrm{N}_t^{(i,j)}$ for $1\leq i,j \leq k$ to decouple the time of the jump in $\overbar{S}_i$ and $\overbar{S}_j$ as described in the figure \ref{fig:auxiliary_limit_clocks}.

\begin{figure}[ht]
  \centering
  \includegraphics[width=.8\textwidth]{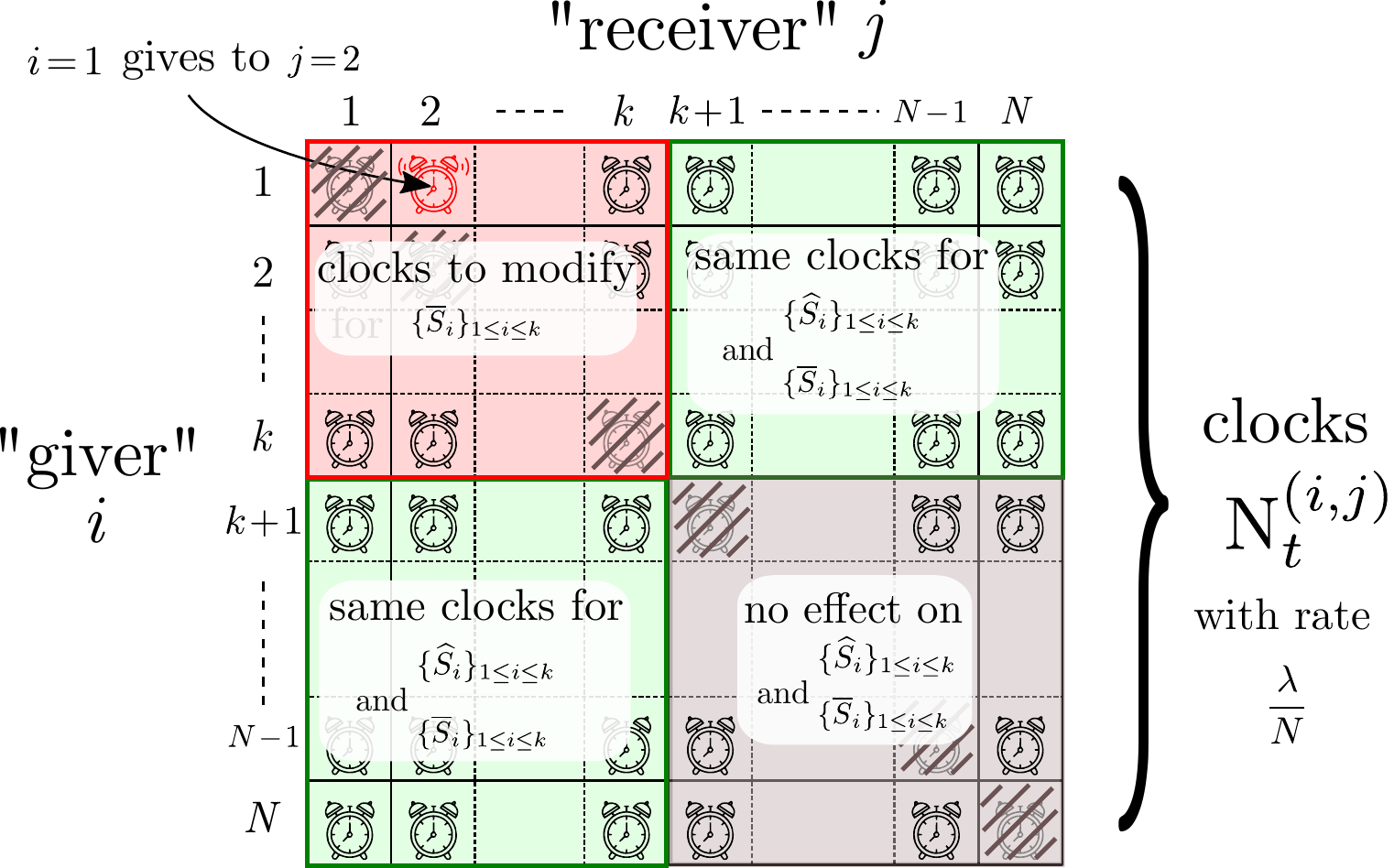}
  \caption{The clocks $\mathrm{N}_t^{(i,j)}$ used to generate the unbiased dynamics \eqref{UEMoriginal} have to be modified to generate the limit dynamics $(\overbar{S}_1(t),\ldots,\overbar{S}_k(t))$ \eqref{UEMlimit_S1}. The processes $\overbar{S}_i(t)$ and $\overbar{S}_j(t)$ have to be independent, thus the clocks $\mathrm{N}_t^{(i,j)}$ for $1\leq i,j \leq k$ cannot be used.}
  \label{fig:auxiliary_limit_clocks}
\end{figure}

\begin{proposition}
  \label{ppo:aux_vs_limit}
  Let $\big(\widehat{S}_1,\ldots,\widehat{S}_N\big)$ solution to \eqref{eq:intermediate} and $\{\overbar{S}_i\}_{1\leq i \leq k}$ independent processes solution to \eqref{UEMlimit_S1}. Then for any fixed $k \in \mathbb N_+$, there exists a coupling such that for all $t\geq 0$:
  \begin{equation}
    \label{eq:particle_aux}
    \dd \mathbbm{E}[|\widehat{S}_i(t)-\overbar{S}_i(t)|] \leq λ\frac{4(k-1)}{N} \,\dd t \quad,\qquad \text{for } 1 \leq i \leq k.
  \end{equation}
\end{proposition}
\begin{proof}
  We assume $i=1$ to simplify the writing. To couple the two processes $\widehat{S}_1$ and $\overbar{S}_1$, we use the same Bernoulli variable $\overbar{Y}(t-)$ to generate both 'receive' dynamics:
  \[\begin{cases}
      \dd \widehat{S}_1(t) &=-\mathbbm{1}_{[1,\infty)}\big(\widehat{S}_1(t-)\big)\dd \widehat{\mathrm{\bf N}}^1_t+\overbar{Y}(t-)\dd \widehat{\mathrm{\bf M}}^1_t,\\
      \dd \overbar{S}_1(t) &=-\mathbbm{1}_{[1,\infty)}\big(\overbar{S}_1(t-)\big)\dd \overbar{\mathrm{\bf N}}^1_t+\overbar{Y}(t-)\dd \overbar{\mathrm{\bf M}}^1_t.
    \end{cases}\]
  Meanwhile, the Poisson clocks $\widehat{\mathrm{\bf N}}_t^1,\,\widehat{\mathrm{\bf M}}_t^1$ are already determined in \eqref{eq:intermediate}:
  \begin{equation}
    \label{eq:clocks_intermediate}
    \widehat{\mathrm{\bf N}}^1_t= \overbar{\mathrm{N}}^{(1,1)}_t + \sum_{j=2}^N \mathrm{N}^{(1,j)}_t \quad \text{and} \quad \widehat{\mathrm{\bf M}}^1_t= \overbar{\mathrm{M}}^{(1,1)}_t + \sum_{j=2}^N \mathrm{N}^{(j,1)}_t.
  \end{equation}
  Unfortunately, we cannot use the same definition for the clocks $\overbar{\mathrm{\bf N}}_t^1$ and $\overbar{\mathrm{\bf M}}_t^1$ as the clocks $\widehat{\mathrm{\bf N}}_t^i$ and $\widehat{\mathrm{\bf M}}_t^j$ are {\it not} independent (they both contain the clock $\mathrm{N}_t^{(i,j)}$). Thus, we need to remove those coupling clocks when defining $\overbar{\mathrm{\bf N}}^1$ and $\overbar{\mathrm{\bf M}}^1$. Fortunately, we only have to generate the dynamics for $k$ process, thus we only have to replace the clocks $\mathrm{N}^{(1,i)}$ and $\mathrm{N}^{(i,1)}$ for $i=1..k$ (see figure \ref{fig:auxiliary_limit_clocks}):
  \begin{equation}
    \label{eq:clocks_bar}
    \overbar{\mathrm{\bf N}}^1_t= \sum_{j=1}^k \overbar{\mathrm{N}}^{(1,j)}_t + \sum_{j=k+1}^N \mathrm{N}^{(1,j)}_t \quad \text{and} \quad \widehat{\mathrm{\bf M}}^1_t= \sum_{j=1}^k \overbar{\mathrm{M}}^{(1,j)}_t + \sum_{j=k+1}^N \mathrm{N}^{(j,1)}_t
  \end{equation}
  where $\overbar{\mathrm{N}}^{(1,j)}_t$ and $\overbar{\mathrm{M}}^{(1,j)}_t$ are independent Poisson clocks with rate $\frac{\lambda}{N}$.

  Using this coupling strategy, the difference $|\widehat{S}_1-\overbar{S}_1|$ could only increase (by $1$) if the clocks $\overbar{\mathrm{N}}^{(1,j)}_t$, $\overbar{\mathrm{M}}^{(1,j)}_t$, $\mathrm{N}^{(1,j)}_t$ or $\mathrm{N}^{(j,1)}_t$ ring for $2\leq j \leq k$ leading to \eqref{eq:particle_aux}.

\end{proof}

Finally, combining propositions \ref{ppo:unbiased_S_hat_S} and \ref{ppo:aux_vs_limit} gives rise to the following theorem.
\begin{theorem}
  \label{proagation_of_chaos}
  Let $\big(S_1,\ldots,S_N\big)$ to be a solution to \eqref{UEMoriginal}. Then for any fixed $k \in \mathbb N_+$ and $t \geq 0$, there exists a coupling between $\big(S_1,\ldots,S_k\big)$ and $\big(\overbar{S}_1,\ldots,\overbar{S}_k\big)$ (with the same initial conditions) such that:
  \begin{equation}
    \label{POC}
    \mathbbm{E}[|S_i(t)-\overbar{S}_i(t)|] \leq \frac{C(t)}{\sqrt{N}} \frac{(\expo^{λt}-1)}{λ} + λ\frac{4(k-1)t}{N}
  \end{equation}
  with $C(t) = \left(\frac{1}{4} + λ 4t\right)^{1/2}  + λ \frac{2}{\sqrt{N}}$ holding for each $1\leq i\leq k$.
\end{theorem}
\begin{proof}
  We assume without loss of generality that $i=1$. First, we show that the processes $S_1$ and $\widehat{S}_1$ remain closed. We denote:
  \begin{displaymath}
    R_i = \mathbbm{1}_{[1,\infty)}\big(S_i\big) \quad,\quad \widehat{R}_i = \mathbbm{1}_{[1,\infty)}\big(\widehat{S}_i\big)\quad,\quad \overbar{R}_i = \mathbbm{1}_{[1,\infty)}\big(\overbar{S}_i\big).
  \end{displaymath}
  We have:
  \begin{eqnarray*}
    \mathbbm{E}[|r-\overbar{r}|] &=& \mathbbm{E}\left[\left|\frac{1}{N} ∑_{i=1}^N R_i-\overbar{r}\right|\right] = \mathbbm{E}\left[\left|\frac{1}{N} ∑_{i=1}^N (R_i - \widehat{R}_i) \;\;+\;\; \frac{1}{N} ∑_{i=1}^N (\widehat{R}_i-\overbar{r})\right|\right] \\
                                 &\leq& \frac{1}{N} ∑_{i=1}^N \mathbbm{E}[|R_i - \widehat{R}_i|] + \mathbbm{E}\left[\left|\frac{1}{N} ∑_{i=1}^N (\widehat{R}_i-\overbar{r})\right|\right]  \\
                                 &\leq& \mathbbm{E}[|S_1 - \widehat{S}_1|] + \mathbbm{E}\left[\left(\frac{1}{N} ∑_{i=1}^N (\widehat{R}_i-\overbar{r})\right)^2\right]^{1/2},
  \end{eqnarray*}
  where we use $|R_i - \widehat{R}_i| \leq |S_i - \widehat{S}_i|$. To control the variance, we expand:
  \begin{eqnarray*}
    \mathbbm{E}\left[\left(\frac{1}{N} ∑_{i=1}^N (\widehat{R}_i-\overbar{r})\right)^2\right] &=& \frac{1}{N}  \mathrm{Var}[\widehat{R}_1] + \frac{N(N-1)}{N^2}\mathrm{Cov}(\widehat{R}_1,\widehat{R}_2)  \\
                                                                                             &\leq& \frac{1}{4N} + \mathrm{Cov}(\widehat{R}_1,\widehat{R}_2),
  \end{eqnarray*}
  since $\widehat{R}_1$ is a Bernoulli variable its variance is bounded by $1/4$. Controlling the covariance of $\widehat{R}_1$ and $\widehat{R}_2$ is more delicate since the two processes are not independent due to the clocks $\mathrm{N}^{(1,2)}_t$ and $\mathrm{N}^{(2,1)}_t$. Fortunately, these clocks have a rate of only $λ/N$ and thus the covariance has to remain small for a given time interval. To prove it, let's use the independent processes $\overbar{R}_1$ and $\overbar{R}_2$:
  \begin{eqnarray*}
    \mathrm{Cov}(\widehat{R}_1,\widehat{R}_2) &=& \mathrm{Cov}(\widehat{R}_1-\overbar{R}_1, \widehat{R}_2-\overbar{R}_2) \leq \left(𝔼[|\widehat{R}_1-\overbar{R}_1|^2]⋅ 𝔼[|\widehat{R}_2-\overbar{R}_2|^2]\right)^{1/2}
  \end{eqnarray*}
  using Cauchy-Schwarz. Since the two processes $\widehat{S}_i$ and $\overbar{S}_i$ remain close, we deduce:
  \begin{eqnarray*}
    𝔼[|\widehat{R}_1(t)-\overbar{R}_1(t)|^2] &=& 𝔼[|\widehat{R}_1(t)-\overbar{R}_1(t)|] \leq 𝔼[|\widehat{S}_1(t)-\overbar{S}_1(t)|] \leq λ \frac{4t}{N},
  \end{eqnarray*}
  using proposition \ref{ppo:aux_vs_limit} (with $k=2$). We conclude that:
  \begin{displaymath}
    \mathbbm{E}[|r(t)-\overbar{r}(t)|] \leq \mathbbm{E}[|S_1(t) - \widehat{S}_1(t)|] + \left(\frac{1}{4N} + λ \frac{4t}{N}\right)^{1/2}.
  \end{displaymath}
  Going back to proposition \ref{ppo:unbiased_S_hat_S}, we find:
  \begin{eqnarray*}
    \dd \mathbbm{E}[|S_i(t)-\widehat{S}_i(t)|] &\leq& λ\,\mathbbm{E}[|S_1(t) - \widehat{S}_1(t)|]\,\dd t + \left(\frac{1}{4N} + λ \frac{4t}{N}\right)^{1/2}\,\dd t \; +\; λ \frac{2}{N}\,\dd t
    \\
                                               &\leq& λ\,\mathbbm{E}[|S_1(t) - \widehat{S}_1(t)|]\,\dd t + \frac{C(t)}{\sqrt{N}}\,\dd t
  \end{eqnarray*}
  with $C(t) = \left(\frac{1}{4} + λ 4t\right)^{1/2}  + λ \frac{2}{\sqrt{N}} = \mathcal{O}(1)$.  Using Gronwall's lemma, since $|S_i(0)-\widehat{S}_i(0)|=0$, we obtain:
  \begin{equation}
    \mathbbm{E}[|S_i(t)-\widehat{S}_i(t)|] \leq \frac{C(t)}{\sqrt{N}} \frac{(\expo^{λt}-1)}{λ}.
  \end{equation}
  We finally conclude by using proposition \ref{ppo:aux_vs_limit} and triangular inequality.
\end{proof}

\begin{remark}
In the community of Markov chains, the process $({\bf S}(t)\colon t\geq 0)$ with ${\bf S}(t):=\left(S_1(t),\ldots,S_N(t)\right)$ can serve as an example a zero-range process \cite{spitzer_interaction_1991}, and it is also observed in \cite{merle_cutoff_2019} that the unbiased exchange model exhibits a \emph{cutoff} phenomenon (see for instance \cite{diaconis_cutoff_1996,aldous_random_1983,aldous_shuffling_1986}), which is now ubiquitous among literatures on interacting Markov chains.
\end{remark}

\subsection{Convergence to equilibrium}

After we achieved the transition from the interacting system of SDEs \eqref{def:unbiased_exchange_model} to the deterministic system of nonlinear ODEs \eqref{eq:law_limit_unbais}, in this section we will analyze \eqref{eq:law_limit_unbais} with the intention of proving convergence of solution of \eqref{eq:law_limit_unbais} to its (unique) equilibrium solution. The main ingredient underlying our proof lies in the reformulation of \eqref{eq:law_limit_unbais} into a (discrete) Fokker-Planck type equation, combined with the standard entropy method \cite{arnold_convex_2001,matthes_entropy_2007,jungel_entropy_2016}. We emphasize here that the convergence of the solution of \eqref{eq:law_limit_unbais} has already been established in \cite{graham_rate_2009,merle_cutoff_2019}, but we include a sketch of our analysis here for the sake of completeness of the present manuscript.

To study the ODE system \eqref{eq:law_limit_unbais}, we introduce some properties of the nonlinear binary collision operator $Q_{unbias}$, whose proof is merely a straightforward calculations and will be omitted.
\begin{lemma}\label{invariant}
If~${\bf p}(t)=\{p_n(t)\}_{n\geq 0}$ is a solution of \eqref{eq:law_limit_unbais}, then
\begin{equation}
\label{eq:conservation_mass_mean_value}
\sum_{n=0}^\infty Q_{unbias}[{\bf p}]_n =0 \quad,\quad \sum_{n=0}^\infty n\,Q_{unbias}[{\bf p}]_n =0.
\end{equation}
In particular, the total mass and the mean value is conserved.
\end{lemma}


Thanks to these conservations, we have ${\bf p}(t) \in V_\mu$ for all $t\geq 0$, where \[V_\mu:=\{{\bf p} \mid \sum_{n=0}^\infty p_n =1,~p_n \geq 0,~\sum_{n=0}^\infty n\,p_n =\mu\}\] is the space of probability mass functions with the prescribed mean value $\mu$. Next, the equilibrium distribution of the limiting dynamics \eqref{eq:law_limit_unbais} is explicitly calculated.

\begin{proposition}
The (unique) equilibrium distribution ${\bf p}^*=\{p^*_n\}_n$ in $V_\mu$ associated with the limiting dynamics \eqref{eq:law_limit_unbais} is given by:
\begin{equation}
\label{eqn:equil_limit_unbias}
p^*_n = p^*_0(1-p^*_0)^n,\quad n\geq 0,
\end{equation}
where $p^*_0 = \frac{1}{1+\mu}$ if we put initially that $\sum_{n=0}^\infty n\,p_n(0)=\mu$ for some $\mu \in \mathbb N_+$.
\end{proposition}
This elementary observation can be verified through straightforward computations, which we will omit here.

Next, we recall the definition of entropy \cite{cover_elements_1999}, which will play a major role in the analysis of the large time behavior of the system \eqref{eq:law_limit_unbais}.

\begin{definition}(\textbf{Entropy})
For a given probability mass function ${\bf p} \in V_\mu$, the entropy of ${\bf p}$ is defined via
\begin{equation*}\label{entropy}
\HH[{\bf p}]=\sum_n p_n \log p_n.
\end{equation*}
\end{definition}

\begin{remark}
It can be readily seen through the method of Lagrange multipliers that the geometric distribution \eqref{eqn:equil_limit_unbias} has the least amount of entropy among probability mass functions from $V_\mu$.
\end{remark}

To prove (strong) convergence of the solution of \eqref{eq:law_limit_unbais} to its equilibrium solution \eqref{eqn:equil_limit_unbias}, a major step is first to realize that the original ODE dynamics \eqref{eq:law_limit_unbais} can be reformulated as a variant of a Fokker-Planck equation \cite{risken_fokker-planck_1996}. Indeed, let us introduce ${\bf s}(t):=\{s_n(t)\}_{n \geq 0}$ with $s_n(t):=p_0(t)[\overbar{r}(t)]^n$. Notice that $\frac{s_{n+1}}{s_n}=\overbar{r}$. Thus, for $n\geq 1$ we can deduce that
\[p'_n=\frac{s_{n+1}}{s_{n+1}}p_{n+1}+\frac{s_n}{s_{n-1}}p_{n-1}-\frac{s_n}{s_n}p_n-\frac{s_{n+1}}{s_n}p_n.\]
Setting $q_n(t) = \frac{p_{n}(t)}{s_n(t)}$, we obtain
\begin{equation}\label{FokkerPlanck}
p'_n = s_{n+1}[q_{n+1}-q_n]-s_n[q_n-q_{n-1}],
\end{equation}
with the convention that $q_{-1}\equiv 1$. This formulation leads to the following:
\begin{proposition}
Let $\{p_n(t)\}_{n\geq 0}$ be the solution to \eqref{eq:law_limit_unbais} and $\varphi \colon \mathbb R \to \mathbb R$ to be a continuous function, then
\begin{equation}\label{IBP}
\sum_{n=0}^\infty p'_n\,\varphi(n)=\sum_{n=0}^\infty (\overbar{r}p_n-p_{n+1})\big(\varphi(n+1)-\varphi(n)\big).
\end{equation}
\end{proposition}

\begin{corollary}
Taking $\varphi(n) \equiv 1$ and $\varphi(n)=n$ for $n\geq 0$ in \eqref{IBP}, we recover the facts that $\sum_{n=0}^\infty p_n$ and $\sum_{n=0}^\infty n\,p_n$ are preserved over time.
\end{corollary}

\noindent Inserting $\varphi(n)=\log p_n$, we can deduce the following important result.

\begin{proposition}[\textbf{Entropy dissipation}]
Let ${\bf p}(t)=\{p_n(t)\}_{n\geq 0}$ be the solution to \eqref{eq:law_limit_unbais} and $\HH[{\bf p}]$ be the associated entropy, then for all $t>0$, \[\frac{\dd}{\dd t}\HH[{\bf p}]= -\mathrm{D}_{\mathrm{KL}} \left( {\bf p}||\mathbf{\tilde{p}} \right) - \mathrm{D}_{\mathrm{KL}} \left( {\bf \tilde{p}}||{\bf p} \right) \leq 0,\] where ${\bf \tilde{p}}:=\{\tilde{p}_n\}_{n\geq 0}$ is defined by $\tilde{p}_0=p_0$ and $\tilde{p}_n=\overbar{r}p_{n-1}$ for $n \geq 1$.
\end{proposition}

\begin{proof}
It is worth noting that \[\sum \limits^{\infty}_{n=0} \tilde{p}_n=p_0+\sum \limits^{\infty}_{n=0} \tilde{p}_{n+1}=p_0+\overbar{r}\sum \limits^{\infty}_{n=0} p_n=p_0+\overbar{r}=1, \] so that $\{\tilde{p}_n\}_n$ indeed defines a probability distribution (for all $t\geq 0$). Then we deduce from \eqref{IBP} that
\begin{align*}
\frac{\dd}{\dd t}\HH[{\bf p}]&=\sum \limits^\infty_{n=0} p'_n \log p_n=\sum \limits^\infty_{n=0} (\overbar{r}p_n-p_{n+1})\log \frac{p_{n+1}}{p_n} \\
&= \sum \limits^{\infty}_{n=0} (p_{n+1}-\tilde{p}_{n+1}) \Big(\log \frac{\overbar{r}p_n}{p_{n+1}}-\log \overbar{r}\Big) \\
&= \sum \limits^{\infty}_{n=0} (p_{n+1}-\tilde{p}_{n+1}) \log \frac{\tilde{p}_{n+1}}{p_{n+1}}  \\
&= \sum \limits^{\infty}_{n=0} (p_{n}-\tilde{p}_{n}) \log \frac{\tilde{p}_{n}}{p_{n}} = -\sum \limits^{\infty}_{n=0} p_n \log \frac{p_n}{\tilde{p}_n}-\sum \limits^{\infty}_{n=0} \tilde{p}_n \log \frac{\tilde{p}_n}{p_n} \\
&= -\mathrm{D}_{\mathrm{KL}} \left({\bf p}||{\bf \tilde{p}} \right) - \mathrm{D}_{\mathrm{KL}} \left({\bf \tilde{p}}||{\bf p} \right) \leq 0,
\end{align*}
in which $\mathrm{D}_{\mathrm{KL}} \left({\bf p}||{\bf q} \right):= \sum_{n=0}^\infty p_n\,\log \frac{p_n}{q_n}$ ($\geq 0$) is the Kullback-Leiber divergence from the probability distribution $\mathbf{q}$ to ${\bf p}$.
\end{proof}

\begin{remark}
By a property of the Kullback-Leiber divergence \cite{csiszar_information_2004}, $\frac{\dd}{\dd t} \HH[{\bf p}]=0$ if and only if ${\bf p}= {\bf \tilde{p}}$, but it can be readily shown that ${\bf p}={\bf \tilde{p}}$ if and only if ${\bf p}$ coincides with the equilibrium distribution ${\bf p}^*$.
\end{remark}

Our next focus is on the demonstration of the strong convergence of solutions ${\bf p}(t)=\{p_n(t)\}_{n\geq 0}$ of \eqref{eq:law_limit_unbais} to its unique equilibrium solution given by \eqref{eqn:equil_limit_unbias}. First of all, we notice that $V_\mu$ is clearly closed and bounded in $\ell^p$ for each $1\leq p\leq \infty$, whence there exists some $\widehat{{\bf p}}=\{\widehat{p}_n\}_{n\geq 0} \in V_\mu$ and a diverging sequence $\{t_k\}_k$ such that ${\bf p}^{(k)}:={\bf p}(t_k) \rightharpoonup \widehat{{\bf p}}$ weakly in $\ell^p$ ($1<p<\infty$) as $k \rightarrow \infty$. In particular, we have the point-wise convergence \[p^{(k)}_n \rightarrow \widehat{p}_n \quad \text{for each $n \geq 0$}.\] Our ultimate goal is to show that $\widehat{{\bf p}} = {\bf p}^*$, for which we first establish the following proposition.

\begin{proposition}\label{continuity}
Suppose that $\{{\bf p}^{(k)}\}_k$ is a sequence of probability distributions in $V_m$ such that \[{\bf p}^{(k)} \rightharpoonup \widehat{p} \] weakly in $\ell^p$ for some $1<p<\infty$. If the family $\{{\bf p}^{(k)}\}_k$ satisfies the following \emph{uniform integrability} condition \cite{oksendal_stochastic_2013}
\begin{equation}\label{UI}
\sum_{n=0}^\infty n^\gamma p^{(k)}_n < \infty \quad \text{uniformly for all $k$}
\end{equation}
for some $\gamma >1$, then \begin{equation}\label{weakconver}
\sum_{n=0}^\infty p^{(k)}_n \log p^{(k)}_n \rightarrow \sum_{n=0}^\infty \widehat{p}_n \log \widehat{p}_n \quad \text{as $k \rightarrow \infty$.}
\end{equation}
\end{proposition}

\begin{proof}
It suffices to show that for any given $\varepsilon >0$, there exists some universal constant $N=N(\varepsilon)$ such that
\begin{equation}\label{unifinte}
\sum_{n=N}^\infty -p^{(k)}_n \log p^{(k)}_n < \varepsilon \quad \forall k\geq 0.
\end{equation}
Assume that $\sum_{n=0}^\infty n^\gamma p^{(k)}_n \leq C$ holds uniformly in $k$ for some constant $\gamma >1$, where $C >0$ is fixed. Then $p^{(k)}_n \leq \frac{C}{n^\gamma}$ for all $n \in \mathbb N$ and $k \in \mathbb N$. Since $g(x):=-x\log x$ is an increasing function for small $x>0$, we have for some fixed sufficiently large $N$  that
\[\sum_{n=N}^\infty -p^{(k)}_n\log p^{(k)}_n \leq \sum_{n=N}^\infty -\frac{C}{n^\gamma}\log \frac{C}{n^\gamma} < \varepsilon,\] and the proof is completed.

\end{proof}

The next lemma ensures that the solution $\{p_n(t)\}_{n\geq 0}$ of our limiting ODE system \eqref{eq:law_limit_unbais} is uniformly integrable (in time), whose proof is elementary and is thus skipped.

\begin{lemma}\label{UIsoln}
Let $\{p_n(t)\}_{n\geq 0}$ to be the solution of \eqref{eq:law_limit_unbais}. Assume that $\sum_{n=0}^\infty p_n(0)a^n < \infty$ for some $a>1$, then for each fixed $\gamma >1$, \[\sum_{n=0}^\infty n^\gamma p_n(t) < \infty \] holds uniformly in time.
\end{lemma}

We are now in a position to prove the desired convergence result.

\begin{proposition}\label{conv_limit_unbias}
The solution ${\bf p}(t)=\{p_n(t)\}_{n\geq 0}$ of \eqref{eq:law_limit_unbais} converges strongly in $\ell^p$ for $1<p<\infty$ as $t\to \infty$ to its unique equilibrium solution ${\bf p}^*=\{p^*_n\}_n$ given by \eqref{eqn:equil_limit_unbias}.
\end{proposition}

\begin{proof}
Our proof follows closely to the general strategy presented in \cite{perko_differential_2013} and is in essence a continuity argument. We will denote the flow of the ODE system \eqref{eq:law_limit_unbais} with initial data ${\bf p}^0 \in V_\mu$  by $\phi_t({\bf p}^0)$. It is recalled that we only need to show that $\widehat{{\bf p}} = {\bf p}^*$. We argue by contradiction and suppose that $\widehat{{\bf p}} \neq {\bf p}^*$. Since $\HH[{\bf p}(t)]$ is strictly decreasing along trajectories of \eqref{eq:law_limit_unbais} and since ${\bf p}(t_k) \rightharpoonup \widehat{{\bf p}}$ weakly in $\ell^p$ ($1<p<\infty$) as $k \rightarrow \infty$, we deduce that $\HH[\phi_{t_k}({\bf p}^0)] \rightarrow \HH[\widehat{{\bf p}}]$ by combining proposition \ref{continuity} and proposition \ref{UIsoln}, whence \[\HH[\phi_{t_k}({\bf p}^0)] > \HH[\widehat{{\bf p}}] \] for all $t>0$. But if $\widehat{{\bf p}} \neq {\bf p}^*$, then for all $s>0$ we must have $\HH[\phi_s(\widehat{{\bf p}})] < \HH[\widehat{{\bf p}}]$, and by continuity, it follows that for all ${\bf p} \in V_\mu$ sufficiently close to $\widehat{{\bf p}}$ in the $\ell^p$ norm ($1<p<\infty$) we have $\HH[\phi_s[{\bf p}]] < \HH[\widehat{{\bf p}}]$ for all $s>0$. But then for ${\bf p}:=\phi_{t_k}({\bf p}^0)$ and sufficiently large $k$, we have \[\HH[\phi_{s+t_k}({\bf p}^0)] < \HH[\widehat{{\bf p}}],\] which contradicts the above inequality. Therefore we must have that $\widehat{{\bf p}} = {\bf p}^*$ and hence ${\bf p}(t) \rightharpoonup {\bf p}^*$ weakly in $\ell^p$ ($1<p<\infty$) as $t \rightarrow \infty$. In particular, we have the pointwise convergence \[p_n(t) \rightarrow p^*_n \quad \text{as $t \rightarrow \infty$ for each $n \geq 0$}.\] Now since
\begin{align*}
\|{\bf p}(t)-{\bf p}^*\|_1:&=\sum_{n=0}^\infty |p_n(t)-p^*_n| = \sum_{n=0}^N |p_n(t)-p^*_n|+\sum_{n=N+1}^\infty |p_n(t)-p^*_n| \\
&\leq \sum_{n=0}^N |p_n(t)-p^*_n|+\sum_{n=N+1}^\infty (p_n(t)+p^*_n),
\end{align*}
by taking $N$ to be sufficiently large and independent of $t$, the desired strong convergence in $\ell^p$ for $1<p<\infty$ follows immediately.
\end{proof}

\section{Poor-biased exchange model}
\label{sec:poor_biased_exchange}
\setcounter{equation}{0}

We now investigate our second model where the 'given' dynamics is biased toward richer agent: the wealthier an agent becomes, the more likely it will give a dollar. As for the previous model, we first investigate the limit dynamics as the number of agents $N$ goes to infinity, then we study the large time behavior and show rigorously the convergence of the wealth distribution to a Poisson distribution.

\subsection{Definition and limit equation}

We use the same setting as the unbiased model: there are $N$ agents with initially the same amount of money $S_i(0)= \mu $ with $\mu ∈ ℕ_+$.

\begin{definition}[\textbf{Poor-biased exchanged model}]
  \label{poor_biased_model}
  The dynamics consists in choosing a ``giver'' $i$ with a probability proportional to its wealth (the wealthier an agent, the more likely it will be a ``giver''). Then it gives one dollar to a ``receiver'' $j$ chosen uniformly.
\end{definition}

From another point of view, the dynamics consist in taking one dollar from the common pot (tax system)  and re-distribute the dollar uniformly among the individuals. Thus instead of `taxing the agents' in the unbiased exchange model, the poor-biased model is `taxing the dollar'. 

The poor-biased model can be written in term of stochastic differential equation, the wealth $S_i$ of agent $i$ evolves according to:
\begin{equation}
  \label{eq:poor_bias}
  \dd S_i(t)= -\sum \limits^{N}_{j=1}  \dd \mathrm{N}^{(i,j)}_{t} + \sum \limits^{N}_{j=1} \dd \mathrm{N}^{(j,i)}_{t},
\end{equation}
with $\mathrm{N}^{(i,j)}_t$ Poisson process with intensity $λ_{i,j}(t)=\frac{λ S_i(t)}{N}$.

Since the clocks  $\{\mathrm{N}^{i,j}_t\}_{1\leq i,j\leq N}$ are now time dependent (in contrast to the unbiased model), the dynamics might appear more difficult to analyze. But it turns out to be simpler, since the rate of receiving a dollar is constant:
\begin{displaymath}
  ∑_{j=1}^N λ_{j,i}(t) = ∑_{j=1}^N\frac{λ S_j(t)}{N} = λ \mu,
\end{displaymath}
where $\mu$ is the (conserved) initial mean. In contrast, in the unbias dynamics, the rate of receiving a dollar is equal to the proportion of rich people $r(t)$ which fluctuates in time. Let's focus on $i=1$ and sum up the clocks introducing:
\begin{equation}
  \label{eq:clock_NM1}
  \mathrm{\bf N}^1_t = \sum_{j=1}^N \mathrm{N}^{(1,j)}_t,\quad \mathrm{\bf M}^1_t = \sum_{j=1}^N \mathrm{N}^{(j,1)}_t,
\end{equation}
where the two Poisson processes $\mathrm{\bf N}^1_t$ and $\mathrm{\bf M}^1_t$ have intensity $λ S_1$ and $λ\mu$ (respectively). Thus, the poor-biased model leads to the equation:
\begin{equation}
  \label{eq:poor_bias2}
  \dd S_1(t)= -\dd \mathrm{\bf N}^{1}_{t} + \dd \mathrm{\bf M}^{1}_t.
\end{equation}
Notice that $S_1(t)$ is not independent of $S_j(t)$ as both processes can jump at the same time due to the two clocks $\mathrm{N}^{(1,j)}_t$ and $\mathrm{N}^{(j,1)}_t$.

Motivated by the equation above, we give the following definition of the limiting dynamics  as $N \rightarrow \infty$.
\begin{definition}(\textbf{Asymptotic Poor-biased model})
  We define $\overbar{S}_1$ to be the compound Poisson process satisfying the following SDE:
  \begin{equation}\label{limit_S}
    \dd \overbar{S}_1(t) = -\dd \overbar{\bf N}^1_t + \dd \overbar{\bf M}^1_t,
  \end{equation}
  in which $\overbar{\bf N}^1_t$ and $\overbar{\bf M}^1_t$ are independent Poisson processes with intensity $\lambda \overbar{S}_1(t)$ and $λ m$ (respectively) where $\mu$ is the mean of $\overbar{S}_1(0)$ (i.e. $\mu =𝔼[\overbar{S}_1(0)]$).
\end{definition}
If we denote by ${\bf p}(t)=\big(p_0(t),p_1(t),\ldots\big)$  the law of the process $\overbar{S}_1(t)$, its time evolution is given by:
\begin{equation}
  \label{eq:law_limit_poor}
  \frac{\dd}{\dd t} {\bf p}(t) = λ \,Q_{poor}[{\bf p}(t)]
\end{equation}
with:
\begin{equation}
  \label{eq:Q_poor}
  Q_{poor}[{\bf p}]_n:= \left\{
    \begin{array}{ll}
      p_1-\mu\,p_0 & \quad \text{if } n=0 \\
      (n+1)p_{n+1} + \mu\,p_{n-1} - (n+\mu)\,p_n & \quad \text{for } n \geq 1
    \end{array}
  \right.
\end{equation}
and $\mu = ∑_{n=0}^{+∞} n\,p_n(t)=∑_{n=0}^{+∞} n\,p_n(0)$.

\subsection{Proof of propagation of chaos}

The aim of this subsection is to prove the propagation of chaos, i.e. that the process $\big(S_1,\ldots,S_k\big)$ converges to $\big(\overbar{S}_1,\ldots,\overbar{S}_k\big)$ as $N$ goes to infinity. As for the unbiased exchange model, the key is to define the Poisson clocks for the limit dynamics $\overline{\mathrm{\bf N}}^i_t$ and $\overline{\mathrm{\bf M}}^i_t$ {\it close to} the clocks of the $N-$particle system  $\mathrm{\bf N}^i_t$ and $\mathrm{\bf M}^i_t$ for $1\leq i \leq k$, but at the same time making the clocks independent. With this aim, we have to 'remove' the clocks $\mathrm{\bf N}^{(i,j)}_t$ and $\mathrm{\bf M}^{(i,j)}_t$ for $1\leq i,j \leq k$.

\begin{theorem}\label{propagation_of_chaos}
  Let $\big(S_1,\ldots,S_N\big)$ to be a solution to \eqref{eq:poor_bias} and $\big(\overbar{S}_1,\ldots,\overbar{S}_k\big)$ a solution to \eqref{limit_S}. Then for any fixed $k \in \mathbb N_+$, there exists a coupling between $\big(S_1,\ldots,S_k\big)$ and $\big(\overbar{S}_1,\ldots,\overbar{S}_k\big)$ (with the same initial conditions) such that:
  \begin{equation}\label{POC}
    \mathbb E[|\overbar{S}_i(t)-S_i(t)|] \leq \frac{4kλ\mu}{N}(\expo^{λt}-1),
  \end{equation}
  holding for each $1\leq i\leq k$.
\end{theorem}


\begin{proof}
  To simplify the writing, we suppose $i=1$. We define for $1\leq i \leq k$ the clocks for the limit dynamics as follow:
  \begin{equation}
    \label{eq:clock_poor_bias_limit}
    \overline{\mathrm{\bf N}}^1_t = \overbar{G}_1 ⋅\left(\sum_{j= k+1}^N \mathrm{N}^{(1,j)}_t\right) +\widehat{\mathrm{N}}^{1}_t \quad,\quad \overline{\mathrm{\bf M}}^1_t =  \left(\sum_{j= k+1}^N \mathrm{N}^{(j,1)}_t\right) +\widehat{\mathrm{M}}^{1}_t.
  \end{equation}
  Here, $\overbar{G}_1$ is a Bernoulli random variable that prevents the clocks to ring for $\overbar{S}_1$ if the rates of the clocks $\mathrm{N}^{(1,j)}_t$ from $k+1\leq j\leq N$ are too large compare to $\overbar{S}_1$. The parameter of this Bernoulli random variable is given by:
  \begin{equation}
    \label{eq:G_R}
    \overbar{G}_1(t) \sim \mathcal B\Big(1\wedge \frac{N\overbar{S}_1(t)}{(N-k)S_1(t)}\Big),
  \end{equation}
  with $a \wedge b= \min\{a,b\}$ for any $a,b∈ ℝ$. On the contrary, the two processes  $\widehat{\mathrm{N}}^{1}_t$ and $\widehat{\mathrm{M}}^{1}_t$ are used to compensate if the rates of the clocks $\mathrm{N}^{(1,j)}_t$ and $\mathrm{N}^{(j,1)}_t$ from $k+1\leq j\leq N$ are not large enough. Both processes $\widehat{\mathrm{N}}^{1}_t$ and $\widehat{\mathrm{M}}^{1}_t$ are independent (inhomogeneous) Poisson processes with rates respectively:
  \begin{equation}
    \label{eq:rate_NM}
    \widehat{μ}(t) = λ\left(\overbar{S}_1(t)-\frac{(N-k)S_1(t)}{N}\right)_+ \quad \text{and} \quad \widehat{ν}(t) = λ\left(\mu-\sum_{j= k+1}^N \frac{S_j(t)}{N}\right)
  \end{equation}
  where $a_+=\max\{a,0\}$ for any $a∈ℝ$. One can check that under the aforementioned setup (coupling of Poisson clocks), $\overline{\mathrm{\bf N}}^1_t$ and $\overline{\mathrm{\bf M}}^1_t$ are indeed independent counting processes with intensity $\lambda \,\overbar{S}_i(t)$ and $λ\,\mu$, respectively. 


  \noindent The difference $|\overbar{S}_1(t)-S_1(t)|$ could increase due to $3$ types of events:
  \begin{itemize}
  \item[i)] $\mathrm{N}^{(1,j)}_t$ and $\mathrm{N}^{(j,1)}_t$ ring for $1\leq j \leq k$,
  \item[ii)] $\widehat{\mathrm{N}}^{1}_t$ and $\widehat{\mathrm{M}}^{1}_t$ ring
  \item[iii)] $\mathrm{N}^{(1,j)}_t$ ring for $j\geq k+1$ and $\overbar{G}_1=0$.
  \end{itemize}
  Notice that the third type of event leads to:
  \begin{equation}
    \label{eq:event_iii}
    S_1(t) = S_1(t-) - 1 \quad,\quad \overbar{S}_1(t) = \overbar{S}_1(t-)
  \end{equation}
  i.e. only $S_1$ gives. However, the event $\{\overbar{G}_1=0\}$ only occurs if $S_1(t-)>\overbar{S}_1(t-)$. Therefore, the event iii) could only make $|\overbar{S}_1(t)-S_1(t)|$ to decay.

  Therefore, we deduce:
  \begin{eqnarray}
    \dd \mathbb{E}[|\overbar{S}_1(t)-S_1(t)|] &\leq&  ∑_{j=1}^k \frac{λ}{N}\mathbb{E}[S_1(t)] \dd t \;+\; ∑_{j=1}^k \frac{λ}{N}\mathbb{E}[S_j(t)] \dd t \nonumber \\
                                              && \quad + 𝔼[\widehat{μ}(t)] \dd t \;+\; 𝔼[\widehat{ν}(t)] \dd t \nonumber \\
                                              &\leq & \frac{2kλ\mu}{N}\dd t \;+\;𝔼[\widehat{μ}(t)] \dd t \;+\; 𝔼[\widehat{ν}(t)] \dd t
                                                      \label{eq:poor_bias_error_decay}
  \end{eqnarray}
  using $𝔼\left[S_j(t)\right]=\mu$ for any $j$. Let's bound the rates $\widehat{μ}$ and $\widehat{ν}$:
  \begin{eqnarray*}
    𝔼[\widehat{μ}] &=&     𝔼\left[λ\Big(\overbar{S}_1-\frac{(N-k)S_1}{N}\Big)_+\right] \;\;\leq \;\;λ 𝔼\left[\big(\overbar{S}_i-S_i\big)_+ \!+\! \frac{kS_1}{N}\right] \\
                   &\leq& λ 𝔼\left[|\overbar{S}_1-S_1|\right] + \frac{\lambda k\mu}{N}  \\
    𝔼[\widehat{ν}] &=&  𝔼\left[λ\Big(\mu-\sum_{j=k+1}^N \frac{S_j}{N}\Big)\right]   = \frac{λk\mu}{N}.
  \end{eqnarray*}
  We deduce from \eqref{eq:poor_bias_error_decay}:
  \begin{equation}
    \label{eq:almost}
    \dd \mathbb{E}[|\overbar{S}_1(t)-S_1(t)|] \leq  λ 𝔼\left[|\overbar{S}_1(t)-S_1(t)|\right]\dd t + \frac{4kλ\mu}{N}\dd t.
  \end{equation}
  Applying the Gronwall's lemma to \eqref{eq:almost} yields the result.

\end{proof}

\subsection{Large time behavior}

After we achieved the transition from the interacting system of SDEs \eqref{eq:poor_bias} to the deterministic system of linear ODEs \eqref{eq:law_limit_poor}, we now analyze the long time behavior of the distribution ${\bf p}(t)$ and its convergence to an equilibrium. The main tool behind proof relies again on the reformulation of \eqref{eq:law_limit_poor} into a (discrete) Fokker-Planck type equation, in conjunction with the standard entropy method \cite{arnold_convex_2001,matthes_entropy_2007,jungel_entropy_2016}.

Let's introduce a function space to study ${\bf p}(t)$:
\begin{eqnarray}
  \label{eq:V_m}
  V_μ&:=&\{ {\bf p} ∈ \ell^2(ℕ) \mid  ∑_{n=0}^∞ p_n =1,~p_n \geq 0,~\sum_{n=0}^\infty n\,p_n =μ\},\\
  \label{eq:domain_Q}
  \mathcal{D}(Q_{poor}) &:=&\{ {\bf p} ∈ \ell^2(ℕ) \mid  Q_{poor}[{\bf p}] ∈\ell^2(ℕ)\},
\end{eqnarray}
where $\ell^2$ denote the vector space of square-summable sequences. In contrast to the unbias model with the dynamics \eqref{eq:law_limit_poor}, the operator $Q_{poor}$ is an unbounded operator (i.e. $\mathcal{D}(Q_{poor}) \not\subset \ell^2(ℕ)$). For any ${\bf p}∈V_μ∩\mathcal{D}(Q_{poor})$, it is straightforward to show that:
\begin{equation}
  \label{eq:conservation_mass_mean_value}
  \sum_{n=0}^\infty Q_{poor}[{\bf p}]_n =0 \quad,\quad \sum_{n=0}^\infty n\,Q_{poor}[{\bf p}]_n =0,
\end{equation}
which express that the total mass and the mean value is conserved. Moreover, there exists a unique equilibrium ${\bf p}^*$ for $Q_{poor}$ in $V_μ$ given by a Poisson distribution:
\begin{equation}
  \label{eq:equil_poor_biased}
  p^*_n = \frac{μ^n}{n!}\expo^{-μ},\quad n\geq 0.
\end{equation}

To investigate the convergence of ${\bf p}(t)$ solution to \eqref{eq:law_limit_poor} to the equilibrium ${\bf p}_*$ \eqref{eq:equil_poor_biased}, we introduce two function spaces.
\begin{definition}
  We define the sub-vector spaces of $\ell^2$:
  \begin{eqnarray}
    \label{eq:H0_H1}
    \mathcal{H}^0&=& \{{\bf p} \in \ell^2(ℕ) \mid \sum_{n=0}^\infty \frac{p^2_n}{p^*_n} < +∞\},\\
    \mathcal{H}^1 &=& \{{\bf p} \in \ell^2(ℕ) \mid \sum_{n=0}^\infty p_n^*\left(\frac{p_{n+1}}{p_{n+1}^*}-\frac{p_n}{p_n^*}\right)^2 < +∞\},
  \end{eqnarray}
  and define corresponding scalar products:
  \begin{equation}
    \label{eq:scalar_prod}
    ⟨{\bf p},\mathbf{q} ⟩_{\mathcal H^0} := \sum_{n=0}^\infty \frac{p_nq_n}{p^*_n} \quad,\quad
    ⟨{\bf p},\mathbf{q} ⟩_{\mathcal H^1} := \sum_{n=0}^\infty p^*_n \left(\frac{p_{n+1}}{p_{n+1}^*}-\frac{p_n}{p_n^*}\right)\left(\frac{q_{n+1}}{p_{n+1}^*}-\frac{q_n}{p_n^*}\right).
  \end{equation}
\end{definition}
The advantage of using the scalar product $⟨.,. ⟩_{\mathcal H^0}$ is that the operator $Q_{poor}$ becomes symmetric. To prove it, we rewrite the operator {\it a la} Fokker-Planck.
\begin{lemma}
  For any ${\bf p}∈ {\mathcal H^0}$, we have:
  \begin{equation}
    \label{eq:FP}
    Q_{poor}[{\bf p}]_n =  μ\,D^-\left(p^*_nD^+\!\left(\frac{p_n}{p^*_n}\right)\right)
  \end{equation}
  with $D^+(p_n) = p_{n+1}-p_n$, $D^-(p_n) = p_n - p_{n-1}$ and the convention $p_{-1}= p_{-1}^*= 0$.
\end{lemma}
\begin{proof}
  Since $p_n^*/p_{n+1}^*=(n+1)/μ$, we find
  \begin{eqnarray*}
    \frac{1}{μ}Q_{poor}[{\bf p}]_n  &=& \frac{p_n^*}{p_{n+1}^*}p_{n+1}-\frac{p_{n-1}^*}{p_{n}^*}p_{n} - \left(\frac{p_n^*}{p_{n}^*}p_{n}-\frac{p_{n-1}^*}{p_{n-1}^*}p_{n-1}\right) \\
                             &=& μp_n^* u_{n+1}- p_{n-1}^* u_{n} \;- \; \left(p_n^*u_{n}-p_{n-1}^*u_{n-1}\right)
  \end{eqnarray*}
  with $u_n=p_n/p_n^*$. Using the notation $D^+$ and $D^-$, we write:
  \begin{eqnarray*}
    \frac{1}{μ}Q_{poor}[{\bf p}]_n  &=& p_n^* D^+u_n - p_{n-1}^* D^+u_{n-1} = D^-(p_n^* D^+u_n).
  \end{eqnarray*}

\end{proof}
\begin{remark}
  Equation \eqref{eq:FP} has a flavor of a Fokker-Planck equation of the form
  \begin{equation}
    \label{eq:FP_continuous}
    ∂_t ρ = ∇⋅\left(ρ_∞ ∇\left(\frac{ρ}{ρ_∞}\right)\right),
  \end{equation}
  where $\rho_\infty$ is an equilibrium distribution to which $\rho$ converges (and $\rho_\infty$ may also depend on $\rho$, making the equation nonlinear).
\end{remark}
As a consequence, we deduce that the operator $Q_{poor}$ is symmetric on ${\mathcal H^0}$.
\begin{proposition}
  For any ${\bf p},\mathbf{q}∈{\mathcal H^0}$, the operator $Q_{poor}$ \eqref{eq:Q_poor} satisfies:
  \begin{equation}
    \label{eq:Q_sym}
    ⟨ Q_{poor}[{\bf p}],\mathbf{q}⟩_{\mathcal H^0} = ⟨ {\bf p},Q_{poor}[\mathbf{q}]⟩_{\mathcal H^0} \qquad \text{for any } {\bf p},\mathbf{q}∈{\mathcal H^0}.
  \end{equation}
  Moreover,
  \begin{equation}
    \label{eq:Q_H1}
    ⟨ Q_{poor}[{\bf p}],{\bf p}⟩_{\mathcal H^0} = -μ\sum_{n=0}^\infty p^*_n\left(D^+\left(\frac{p_n}{p^*_n}\right) \right)^2 = -μ\|{\bf p}\|_{\mathcal H^1}^2.
  \end{equation}
\end{proposition}
\begin{proof}
  We simply use integration by parts:
  \begin{align*}
    \frac{1}{μ}⟨ Q_{poor}[{\bf p}],\mathbf{q}⟩_{\mathcal H^0} &= \sum_{n=0}^∞D^-\left(p^*_nD^+\frac{p_n}{p^*_n}\right) \frac{q_n}{p^*_n} = -\sum_{n=0}^\infty p^*_n\left(D^+\frac{p_n}{p^*_n}\right)\left(D^+\frac{q_n}{p^*_n}\right) \\
                                                          &= \sum_{n=0}^∞ \frac{p_n}{p^*_n}D^-\left(p^*_nD^+\frac{q_n}{p^*_n}\right)  =\frac{1}{μ}⟨ {\bf p},Q_{poor}[\mathbf{q}]⟩_{\mathcal H^0}.
  \end{align*}

\end{proof}
Furthermore, the operator $-Q_{poor}$ would have a so-called spectral gap if one can show that the norm $\|.\|_{\mathcal H^1}$ controls the norm $\|.\|_{\mathcal H^0}$. To prove it, we establish a Poincaré inequality.
\begin{lemma}
  \label{lem:Poincare}
  There exists a constant $C_p>0$ such that for any ${\bf p} ∈{\mathcal H^1}$ satisfying $∑_n p_n=1$
  \begin{equation}
    \label{eq:poincare}
    \|{\bf p}-{\bf p}^*\|_{\mathcal H^0}^2  \leq C_p\|{\bf p}\|_{\mathcal H^1}^2
  \end{equation}
  where $\|.\|_{\mathcal H^0}$ and $\|.\|_{\mathcal H^1}$ are defined in \eqref{eq:scalar_prod} and ${\bf p}^*$ is the equilibrium \eqref{eq:equil_poor_biased}.
\end{lemma}
\begin{proof}
  Similar to the standard proof of a classical Poincar\'e inequality we proceed by contradiction. Assume that no such $C_p$ exists, then there exists a sequence (of sequence) ${\bf p}^{(k)}$ such that $\sum_{n=0}^\infty p^{(k)}_n = 1$ and
  \begin{equation}
    \label{contra}
    \|{\bf p}^{(k)}-{\bf p}^*\|_{\mathcal H^0}  \geq k \|{\bf p}^{(k)}\|_{\mathcal H^1}
  \end{equation}
  for all $k \in \mathbb N$. Denote ${\bf s}^{(k)} = {\bf p}^{(k)} - {\bf p}^*_n$. Then we have $\sum_{n=0}^\infty s^{(k)}_n = 0$ and \eqref{contra} reads
  \begin{equation}
    \label{contra2}
    \|{\bf s}^{(k)}\|_{\mathcal H^0} \geq k \|{\bf s}^{(k)}\|_{\mathcal H^1}.
  \end{equation}
  Without loss of generality, we can assume the normalization condition $\|{\bf s}^{(k)}\|_{\mathcal H^0}=1$ for all $k$ and thus $\|{\bf s}^{(k)}\|_{\mathcal H^1}\leq \frac{1}{k}$. By weak compactness, there exists ${\bf s}^\infty \in \mathcal H^0$ such that ${\bf s}^{(k)} \rightharpoonup {\bf s}^\infty$ in $\mathcal H^0$ and in particular $s^{(k)}_n \xrightarrow[]{k\to \infty} s^{\infty}_n$ for all $n$.\\
  Since $\|{\bf s}^{(k)}\|_{\mathcal H^1}\leq \frac{1}{k}$, we also have $\frac{s^{(k)}_{n+1}}{p^*_{n+1}}-\frac{s^{(k)}_n}{p^*_n} \xrightarrow[]{k\to \infty} 0$ for all $k$, or equivalently, $(n+1)s^N_{n+1}-μs^N_n \xrightarrow[]{k\to \infty} 0$. Thus, $(n+1)s^\infty_{n+1} = μs^\infty_n$ and therefore $s^\infty_n = \frac{μ^n}{n!}s^\infty_0$ for all $n$. As $\sum_{n=0}^\infty s^\infty_n = 0$, we must have $s^\infty_0=0$ and therefore ${\bf s}^∞=0$. Contradiction, $\|{\bf s}^∞\|_{\mathcal H^0}=1$ since $\|{\bf s}^{(k)}\|_{\mathcal H^0}=1$ for all $k$.

\end{proof}
As a result of the lemma, the operator $-Q_{poor}$ has a spectral gap of at least $1/C_p$ since:
\begin{equation}
  \label{eq:spectral_gap}
  ⟨-Q_{poor}[{\bf p}-{\bf p}_∞]\,,\,{\bf p}-{\bf p}_∞⟩_{\mathcal H^0} = ⟨-Q_{poor}[{\bf p}]\,,\,{\bf p}⟩_{\mathcal H^0} = \|{\bf p}\|_{\mathcal H^1}^2 \geq \frac{1}{C_p} \|{\bf p}-{\bf p}^*\|_{\mathcal H^0}^2.
\end{equation}

We shall establish the existence of a unique global solution to the linear ODE system \eqref{eq:law_limit_poor}. The key ingredient in our proof relies heavily on standard theory of maximal monotone operators (see for instance Chapter 7 of \cite{brezis_functional_2010}).
\begin{proposition}
  \label{wellposedness}
  Given any ${\bf p}_0 \in \mathcal{D}(Q_{poor})$, there exists a unique function
  \begin{displaymath}
    {\bf p}(t) \in C^1\big([0,\infty); \mathcal H^0) \cap C\big([0,\infty); \mathcal{D}(Q_{poor})\big)
  \end{displaymath}
  satisfying \eqref{eq:law_limit_poor}.
\end{proposition}
\begin{proof}
  We use the Hille-Yosida theorem and show that the (unbounded) linear operator $-Q_{poor}$ on $\mathcal H^0$ is a maximal monotone operator. The monotonicity of $-Q_{poor}$ follows from its symmetric property on $\mathcal H^0$:
  \begin{displaymath}
    \langle -Q_{poor}[\mathbf{v}],\mathbf{v} \rangle_{\mathcal H^0} = μ\sum_{n=0}^\infty p^*_n\left(D^+\left( \frac{v_n}{p^*_n}\right)\right)^2 \geq 0 \quad \text{for all } \mathbf{v} \in \mathcal D(Q_{poor}).
  \end{displaymath}
  To show the maximality of $-Q_{poor}$, it suffices to show $R(I-Q_{poor}) = \mathcal H^0$, i.e., for each $\mathbf{f} \in \mathcal H^0$, the equation ${\bf p}-Q_{poor}[{\bf p}] = \mathbf{f}$ admits at least one solution ${\bf p} \in \mathcal D(-Q_{poor})$. To this end, the weak formulation of ${\bf p}-Q_{poor}[{\bf p}] = \mathbf{f}$ reads
  \begin{equation}
    \label{weak_formulation}
    \langle {\bf p},\mathbf{q} \rangle_{\mathcal H^0} + \langle -Q_{poor}[{\bf p}],\mathbf{q} \rangle_{\mathcal H^0} = \langle \mathbf{f},\mathbf{q} \rangle_{\mathcal H^0} \quad \text{for all }~\mathbf{q} \in \mathcal H^0,
  \end{equation}
  whence the Lax-Milgram theorem yields a unique ${\bf p} \in \mathcal H^1$.

\end{proof}
We can now prove the convergence of ${\bf p}(t)$ solution of \eqref{eq:law_limit_poor} to its equilibrium solution \eqref{eq:equil_poor_biased}.

\begin{theorem}
  \label{thm:Model2_expo_decay}
  Let ${\bf p}(t)$ be the solution of \eqref{eq:law_limit_poor} and ${\bf p}^*$ the corresponding equilibrium. Then:
  \begin{equation}
    \label{expodecay}
    \|{\bf p}(t)\!-\!{\bf p}^*\|_{\mathcal H^0} \leq \|{\bf p}_0\!-\!{\bf p}^*\|_{\mathcal H^0}\expo^{-\frac{λ}{C_p}t}
  \end{equation}
  where ${\bf p}_0$ is the initial condition, i.e. ${\bf p}(t=0)={\bf p}_0$.
\end{theorem}
\begin{proof}
  Taking the derivative of the square norm gives:
  \begin{eqnarray}
    \nonumber
    \frac{1}{2}\frac{d}{dt} \|{\bf p}(t)-{\bf p}^*\|_{\mathcal H^0}^2 &=& ⟨{\bf p}'(t)\,,\,{\bf p}(t)-{\bf p}^*⟩_{\mathcal H^0} = λ⟨Q_{poor}[{\bf p}(t)]\,,\,{\bf p}(t)-{\bf p}^*⟩_{\mathcal H^0}\\
                                                                            &=& λ⟨{\bf p}(t)\,,\,Q_{poor}[{\bf p}(t)]⟩_{\mathcal H^0} = -λ \|{\bf p}(t)\|_{\mathcal H^1}^2,  \label{eq:poor_bias_dt_l2}
  \end{eqnarray}
  using the symmetry of $Q_{poor}$ and the relation \eqref{eq:Q_H1}. Using the Poincaré constant from lemma \eqref{lem:Poincare}, we deduce:
  \begin{eqnarray*}
    \frac{1}{2}\frac{d}{dt} \|{\bf p}(t)-{\bf p}^*\|_{\mathcal H^0}^2 &\leq& -\frac{λ}{C_p}\|{\bf p}(t)-{\bf p}^*\|_{\mathcal H^0}^2.
  \end{eqnarray*}
  Applying the Gronwall's lemma leads to the result.

\end{proof}

To finish our investigation of the poor-biased dynamics, we would like to find an explicit rate for the decay of the solution ${\bf p}(t)$ toward the equilibrium  ${\bf p}^*$, i.e. find an explicit value for the Poincaré constant $C_p$ in lemma \eqref{lem:Poincare}. The key idea, due to Bakry and Emery \cite{bakry_diffusions_1985}, is to compute the second time derivative of $\|{\bf p}(t)-{\bf p}^*\|_{\mathcal H^0}$.

\begin{lemma}
  \label{lem:Q_Q}
  For any ${\bf p}∈V_μ ∩ \mathcal{D}(Q_{poor})$, we have:
  \begin{equation}
    \label{eq:Q_Q}
    ⟨Q_{poor}[Q_{poor}[{\bf p}]]\,,\,{\bf p}⟩_{\mathcal H^0} \geq -μ ⟨Q_{poor}[{\bf p}]\,,\,{\bf p}⟩_{\mathcal H^0}.
  \end{equation}
\end{lemma}
\begin{proof}
  Using the symmetry of $Q_{poor}$, we have:
  \begin{eqnarray*}
    ⟨Q_{poor}[Q_{poor}[{\bf p}]]\,,\,{\bf p}⟩_{\mathcal H^0} &=& ⟨Q_{poor}[{\bf p}]\,,\,Q_{poor}[{\bf p}]⟩_{\mathcal H^0} = μ^2∑_{n=0}^∞ D^-\left(p^*_n \,z_n\right) D^-\left(p^*_n\,z_n\right) \frac{1}{p_n^*}
  \end{eqnarray*}
  with $z_n=D^+\!\left(\frac{p_n}{p^*_n}\right)$. Since $D^-\left(p^*_n\,z_n\right)\frac{1}{p_n^*} = μz_n- nz_{n-1}$,
  integration by parts gives:
  \begin{eqnarray*}
    \frac{1}{μ^2}⟨Q_{poor}[Q_{poor}[{\bf p}]]\,,\,{\bf p}⟩_{\mathcal H^0} &=& ∑_{n=0}^∞ D^-\left(p^*_n \,z_n\right)(z_n- \frac{n}{μ}z_{n-1}) \\
                                                            &=& ∑_{n=0}^∞ D^-\left(p^*_n \,z_n\right)z_n  \;+\;  ∑_{n=0}^∞  p^*_n \,z_n D^+\left(\frac{n}{μ}z_{n-1}\right).
  \end{eqnarray*}
  Using $D^+(\frac{n}{μ}z_{n-1}) =  \frac{n}{μ}D^-(z_n) + \frac{z_n}{μ}$, we deduce:
  \begin{eqnarray*}
    \frac{1}{μ^2}⟨Q_{poor}[Q_{poor}[{\bf p}]]\,,\,{\bf p}⟩_{\mathcal H^0} &=& ∑_{n=0}^∞ D^-\left(p^*_n \,z_n\right)z_n  \;+\;  ∑_{n=0}^∞  p^*_n \,z_n \frac{n}{μ}D^-(z_n) + ∑_{n=0}^∞  p^*_n \frac{z_n^2}{μ} \\
                                                            &=& ∑_{n=0}^∞ \Big(D^-\left(p^*_n \,z_n\right)z_n  +p^*_n \,z_n \frac{n}{μ}D^-( z_n)\Big) + \frac{1}{μ} ⟨Q_{poor}[{\bf p}]\,,\,{\bf p}⟩_{\mathcal H^0} \\
                                                            &=:& A \;+\; \frac{1}{μ} ⟨Q_{poor}[{\bf p}]\,,\,{\bf p}⟩_{\mathcal H^0}
  \end{eqnarray*}
  To conclude we have to show that $A\geq 0$. Notice that $p^*_n  \frac{n}{μ} = p_{n-1}^*$, thus:
  \begin{eqnarray*}
    A &=& ∑_{n=0}^∞ \Big(D^-\left(p^*_n \,z_n\right)z_n  +p^*_{n-1} \,z_n D^-(z_n)\Big) \\
      &=& -∑_{n=0}^∞ p^*_n \,z_nD^+(z_n)  + ∑_{n=0}^∞ p^*_{n} \,z_{n+1} D^-(z_{n+1}) \\
      &=& ∑_{n=0}^∞ p^*_n  \Big(-z_nD^+(z_n)+z_{n+1} D^-(z_{n+1})\Big) \\
      &=& ∑_{n=0}^∞ p^*_n  \Big(z_{n+1}-z_{n}\Big)^2\geq 0.
  \end{eqnarray*}

\end{proof}

\begin{remark}
In general, the computation of the second time derivative of the energy (in our case, $\|{\bf p}(t)-{\bf p}^*\|_{\mathcal H^0}$) requires a number of smartly-chosen integration by parts. However, these computations can actually be made more tractable and organized to some extent. We refer interested readers to \cite{matthes_convex_2011,jungel_algorithmic_2006,bukal_entropies_2011} for ample illustration of the technique known as systematic integration by parts.
\end{remark}

\begin{proposition}
  \label{ppo:decay_poor_bias}
  The exponential decay rate in theorem \eqref{thm:Model2_expo_decay} is at least $λ$, i.e. $\|{\bf p}(t)\!-\!{\bf p}^*\|_{\mathcal H^0} \leq C\expo^{-λt}$.
\end{proposition}
\begin{proof}
  Taking the second derivative and using the symmetry of $Q_{poor}$ give:
  \begin{eqnarray*}
    \frac{1}{2}\frac{d^2}{dt^2} \|{\bf p}(t)-{\bf p}^*\|_{\mathcal H^0}^2 &=& \frac{d}{dt} λ⟨Q_{poor}[{\bf p}(t)]\,,\,{\bf p}(t)⟩_{\mathcal H^0} = 2λ^2⟨Q_{poor}[{\bf p}(t)]\,,\,Q_{poor}[{\bf p}(t)]⟩_{\mathcal H^0}  \\
                                                                                &\geq& -2λ^2μ ⟨Q_{poor}[{\bf p}]\,,\,{\bf p}⟩_{\mathcal H^0} = -λ \frac{d}{dt} \|{\bf p}(t)-{\bf p}^*\|_{\mathcal H^0}^2
  \end{eqnarray*}
  thanks to \eqref{eq:Q_Q} and  \eqref{eq:poor_bias_dt_l2}. Denoting $ϕ(t)=\|{\bf p}(t)-{\bf p}^*\|_{\mathcal H^0}^2$, we have: $ϕ'' \geq -2λ ϕ'$.  Integrating over the interval $(t,+∞)$ yields:
  \begin{displaymath}
    0-ϕ'(t) \geq -2λ (0-ϕ(t)) \quad ⇒ \quad ϕ'(t) \leq -2λ ϕ(t)
  \end{displaymath}
  and the Gronwall's lemma allows to obtain our result. \\
  It remains to justify that $\lim_{t→+∞} ϕ'(t)= \lim_{t→+∞} ϕ(t) =0$. Theorem \eqref{thm:Model2_expo_decay} already shows that $\lim_{t→+∞} ϕ(t)=0$. Moreover, denoting $g(t)=-ϕ'(t)\geq 0$, we have $g'\leq -2λg$. Thus, by Gronwall's lemma, $g(t) \stackrel{t → +∞}{⟶} 0$.

\end{proof}


\subsection{Numerical illustration poor-biased model}

We investigate numerically the convergence of ${\bf p}(t)$ solution to the poor-biased model \eqref{eq:law_limit_poor} to the equilibrium distribution ${\bf p}^*$ \eqref{eq:equil_poor_biased}. We use $μ=5$ (average money) and $λ=1$ (rate of jumps) for the model. To discretize the model, we use $1,001$ components to describe the distribution ${\bf p}(t)$ (i.e. $(p_0(t),…,p_{1000}(t))$). As initial condition, we use $p_μ(0)=1$ and $p_i(0)=0$ for $i \neq μ$. The standard Runge-Kutta fourth-order method (e.g. RK4) is used to discretize the ODE system \eqref{eq:law_limit_poor} with  the time step $Δt=0.01$.

We plot in figure \eqref{fig:model_decay_poor_biased}-left the numerical solution ${\bf p}$ at $t=12$ unit time and compare it to the equilibrium distribution ${\bf p}^*$. The two distributions are indistinguishable. Indeed, plotting the evolution of the difference $\|{\bf p}(t)-{\bf p}^*\|_{\mathcal H^0}$ (figure \eqref{fig:model_decay_poor_biased}-right) shows that the difference is already below $10^{-10}$. Moreover, the decay is clearly exponential as we use semi-logarithmic scale.

\begin{figure}[!htb]
  \begin{subfigure}{0.45\textwidth}
    \centering
    \includegraphics[scale=0.4]{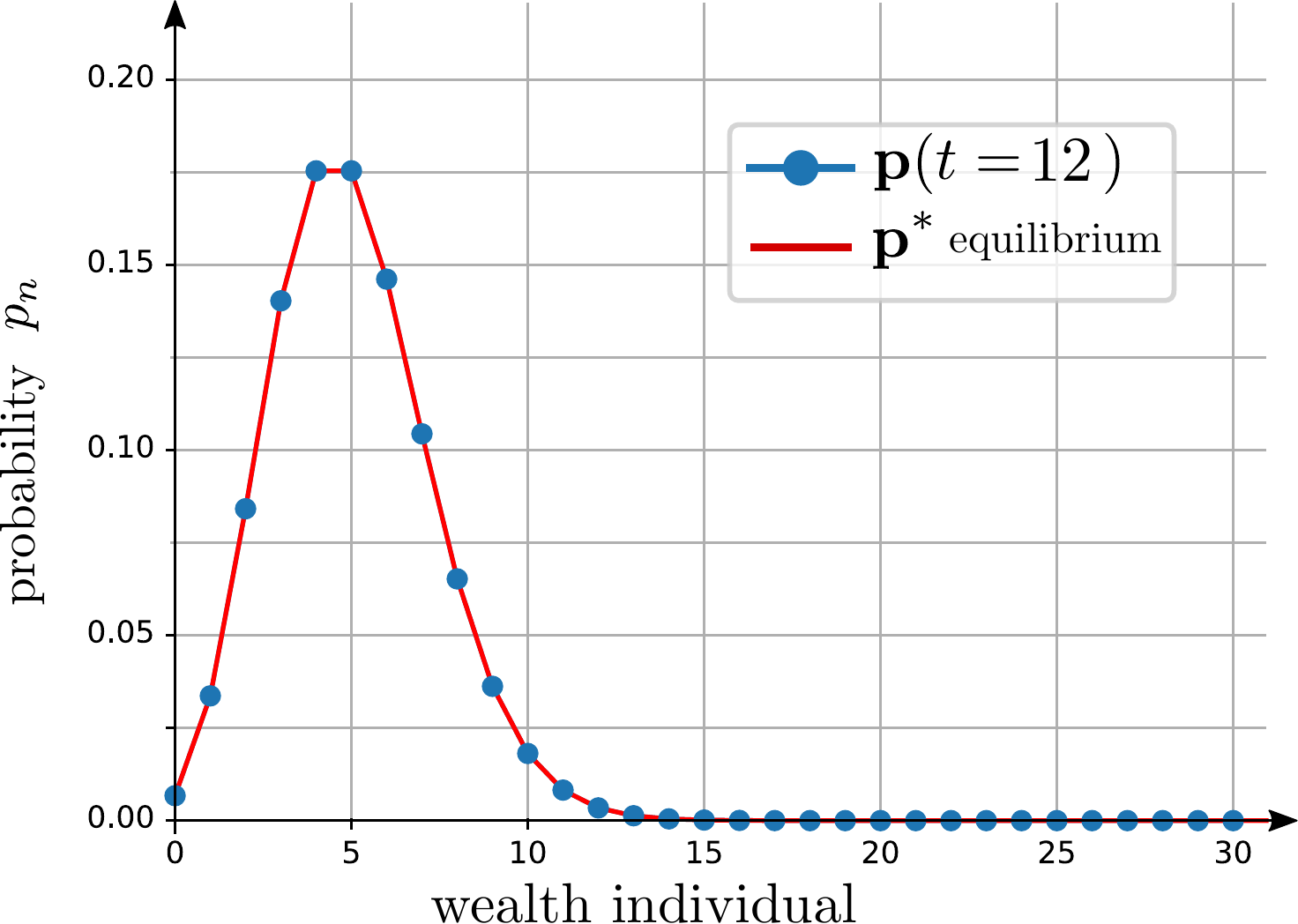}
  \end{subfigure}
  \hspace{0.1in}
  \begin{subfigure}{0.45\textwidth}
    \centering
    \includegraphics[scale=0.4]{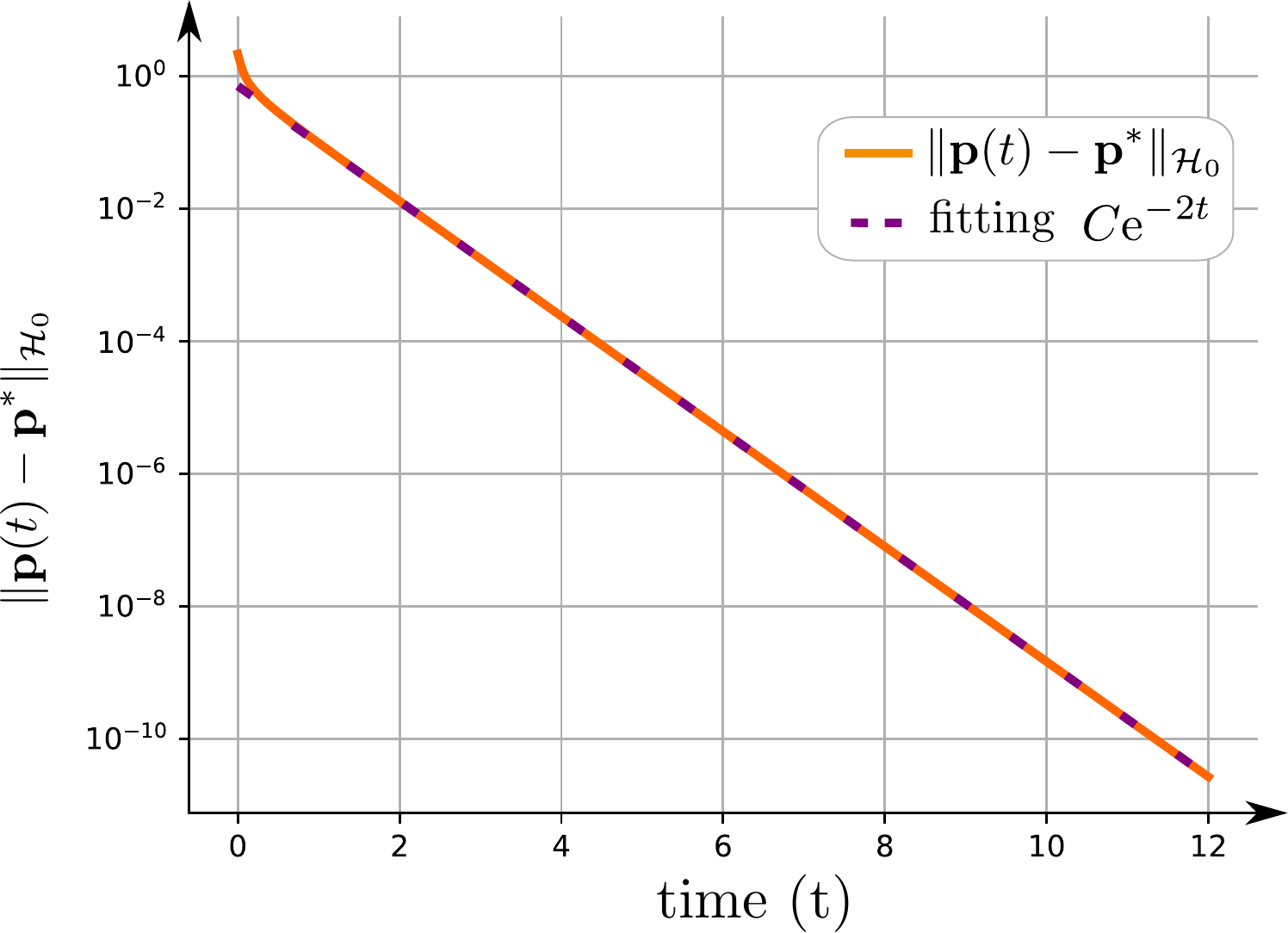}
  \end{subfigure}
  \caption{{\bf Left}: comparison between the numerical solution ${\bf p}(t)$ \eqref{eq:law_limit_poor} of the poor-bias model and the equilibrium ${\bf p}^*$ \eqref{eq:equil_poor_biased}. The two distributions are indistinguishable. {\bf Right}: decay of the difference $\|{\bf p}(t)-{\bf p}^*\|_{\mathcal H^0}$ in semilog scale. The decay is exponential as predicted by the theorem  \ref{thm:Model2_expo_decay}.}
  \label{fig:model_decay_poor_biased}  
\end{figure}

Notice that the numerical simulation suggests that the optimal decay rate of $\|{\bf p}(t)-{\bf p}^*\|_{\mathcal H^0}$ is $2λ$, which is twice the analytical decay rate $λ$ proved in proposition \ref{ppo:decay_poor_bias}. The reason for this discrepancy is that the solution of ${\bf p}(t)$ remains in the subspace $V_μ ∩ \mathcal{D}(Q_{poor})$, i.e. the mean of ${\bf p}(t)$ is preserved. The analysis of the spectral gap of $Q_{poor}$ in the proposition \ref{ppo:decay_poor_bias} does not take account this constraint.

We numerically investigate the spectrum of $-Q_{poor}$ denoted $\{α_n\}_{n=1}^∞$. The first eigenvalue satisfies $α_1=0$ due to the equilibrium ${\bf p}^*$ (i.e. $Q_{poor}[{\bf p}^*]=0$). The other eigenvalues are $α_n=n-1$ and in particular the spectral gap is $α_2=1$. One can find explicitly a corresponding eigenfunction given by:
\begin{equation}
  \label{eq:eigen2}
  {\bf p}^{(2)} = D^-({\bf p}^*) = (p^*_0,p^*_1-p^*_0,\ldots,p^*_n-p^*_{n-1},\ldots).
\end{equation}
Thus, for any ${\bf p}∈V_μ ∩ \mathcal{D}(Q_{poor})$, we find:
\begin{displaymath}
  ⟨{\bf p}\,,\,{\bf p}^{(2)}⟩_{\mathcal H^0} = ∑_{n=0}^∞ p_n (p^*_n-p^*_{n-1})\frac{1}{p^*_n} = ∑_{n=0}^∞ p_n(1-n/μ) = 1 - μ/μ = 0.
\end{displaymath}
This explains why the effective spectral gap for the dynamics is given by $α_3$ and not $α_2$: the solution ${\bf p}(t)$ \eqref{eq:law_limit_poor} lives in $V_μ ∩ \mathcal{D}(Q_{poor})$ and therefore it is orthogonal to ${\bf p}^{(2)}$.

\begin{remark}
  We can find explicitly the exact formulation of the eigenfunction ${\bf p}^{(k)}$ of $-Q_{poor}$ for all $k \in \mathbb N_+$. We find by induction:
  \begin{equation}
    \label{eq:eigen_k}
    {\bf p}^{(k)} = \left(p^*_0, p^*_1 - (k-1)p^*_0,\cdots, p^*_n + \sum_{j=0}^{n-1}(-1)^{n-j} \frac{\prod_{\ell=1}^{n-j} (k-\ell)}{(n-j)!}p^*_j, \cdots\right)
  \end{equation}
  leading to:
  \begin{equation}
    \label{eq:eigen_k_pn}
    p^{(k)}_n = \sum_{j=0}^n \binom{k-1}{j}(-1)^j \frac{μ^{n-j}}{(n-j)!}\expo^{-μ},\quad n\geq 0,
  \end{equation}
  with $\binom{k}{j}$ binomial coefficient (i.e. $\binom{k}{j}=\frac{k!}{(k-j)!\,j!}$). Moreover, through an induction argument and some combinatorial identities, we can verify that $\langle {\bf p}^{(m)}, {\bf p}^{(k)} \rangle_{\mathcal H^0} = 0 $ for $m\neq k$. We speculate that $\{{\bf p}^{(k)}\}_{k=1}^\infty$ spans the entire space $\mathcal H^0$, but we do not have a proof for this conjecture.
\end{remark}

\section{Rich-biased exchange model}
\label{sec:rich_biased_exchange}
\setcounter{equation}{0}

In our third model, the selection of the 'giver' is biased toward the poor instead of the rich, i.e. the more money an individual has the less likely it will be chosen.

\subsection{Definition and limit equation}

As before, the definition of the model is given first.
\begin{definition}(\textbf{Rich-biased exchange model})
  \label{PEM}
  A ``giver'' $i$ is chosen with inverse proportionality of its wealth. The ``receiver'' $j$ is chosen uniformly.
\end{definition}
The rich-biased model leads to the following stochastic differential equation:
\begin{equation}
  \label{eq:rich_bias}
  \dd S_i(t)= -\sum \limits^{N}_{j=1}  \dd \mathrm{N}^{(i,j)}_{t} + \sum \limits^{N}_{j=1} \dd \mathrm{N}^{(j,i)}_{t},
\end{equation}
with $\mathrm{N}^{(i,j)}_t$ Poisson process with intensity $λ_{ij}$ given by:
\begin{equation}
  \label{eq:rate_richbiased}
  λ_{ij}=\left\{
    \begin{array}{cl}
      0 & \text{ if } S_i=0 \\
      \frac{λ}{N}⋅\frac{1}{S_i} & \text{ if } S_i>0
    \end{array}
  \right.
\end{equation}
An agent $i$ receives a dollar at rate $λw$ where $w$ is the inverse of the harmonic mean:
\begin{equation}
  \label{eq:harmonic_mean}
  w= \frac{1}{N}∑_{S_k>0} \frac{1}{S_k}.
\end{equation}



\begin{definition}(\textbf{Asymptotic Rich-biased model})
  \begin{equation}\label{limit_Rich_biased}
    \dd \overbar{S}_1(t) = -\dd \overbar{\bf N}^1_t + \dd \overbar{\bf M}^1_t,
  \end{equation}
  in which $\overbar{\bf N}^1_t$ and $\overbar{\bf M}^1_t$ are independent Poisson processes with intensity $λ/\overbar{S}_1(t)$ (if $\overbar{S}_1(t)>0$) and $λ\overbar{w}(t)$ respectively. The inverse mean $\overbar{w}(t)$ is given by:
  \begin{equation}
    \label{eq:harmonic_average}
    \overline{w}[{\bf p}(t)] := \sum_{n =1}^\infty \frac{p_n(t)}{n}
  \end{equation}
  where ${\bf p}(t)=\big(p_0(t),p_1(t),\ldots\big)$  the law of the process $\overbar{S}_1(t)$. The time evolution of ${\bf p}(t)$ is given by:
  \begin{equation}
    \label{eq:law_limit_rich}
    \frac{\dd}{\dd t} {\bf p}(t) = λ \,Q_{rich}[{\bf p}(t)]
  \end{equation}
  with:
  \begin{equation}
    \label{eq:Q_rich}
    Q_{rich}[{\bf p}]_n:= \left\{
      \begin{array}{ll}
        p_1-\overline{w}\,p_0 & \quad \text{if } n=0 \\
        \frac{p_{n+1}}{n+1} + \overline{w}p_{n-1} - \left(\frac{1}{n}+\overline{w}\right)p_n & \quad \text{for } n \geq 1
      \end{array}
    \right.
  \end{equation}
\end{definition}
We will also need the weak form of the operator: for any test function $φ$:
\begin{equation}
  \label{eq:Q_rich_weak}
  ⟨Q_{rich}[{\bf p}]\,,\,φ⟩ = ∑_{n\geq 0}  p_n\left(\overline{w}φ(n+1)+\frac{\mathbbm{1}_{\{n\geq 1\}}}{n}φ(n-1) - \Big(\overline{w}+\frac{\mathbbm{1}_{\{n\geq 1\}}}{n}\Big)φ(n)\right)
\end{equation}

\subsection{Propagation of chaos using empirical measure}

We investigate the propagation of chaos for the rich-biased dynamics using the empirical measure (see subsection \ref{sec:empirical_measure}). We consider $\{S_i(t)\}_{1\leq i \leq N}$ the solution to \eqref{eq:rich_bias} and introduce the empirical measure:
\begin{equation}
  \label{eq:empirical}
  {\bf p}_{emp}(t) = \frac{1}{N} ∑_{i=1}^N δ_{S_i(t)}(s).
\end{equation}
The goal is to show that the stochastic measure ${\bf p}_{emp}(t)$ converges to the deterministic density ${\bf p}(t)$ solution of \eqref{eq:law_limit_rich}. The main difficulty is that the empirical measure is a stochastic process on a Banach space $\ell^1(ℕ)$ and thus of infinite dimension. Fortunately, the space is a discreet (i.e. $ℕ$) and therefore we do not have to consider stochastic partial differential equations which are famously difficult. Moreover, we only have to consider a finite number of possible jumps.

When agent $i$ gives a dollar to $j$ (i.e. $(S_i,S_j)\;\begin{tikzpicture} \draw [->,decorate,decoration={snake,amplitude=.4mm,segment length=2mm,post length=1mm}] (0,0) -- (.6,0);\end{tikzpicture}\;(S_i-1,S_j+1)$), the empirical measure is transformed as
\begin{equation}
  \label{eq:jump_emp}
  {\bf p}_{emp}\;\;\begin{tikzpicture} \draw [->,decorate,decoration={snake,amplitude=.4mm,segment length=2mm,post length=1mm}]
    (0,0) -- (.6,0);\end{tikzpicture}\;\; {\bf p}_{emp} + \frac{1}{N}\Big(δ_{S_i-1} + δ_{S_j+1}- δ_{S_i}- δ_{S_j}\Big).
\end{equation}
To write down the evolution equation satisfied by ${\bf p}_{emp}$, we regroup the agents with the same number of dollars (i.e. we project the dynamics on a subspace).

\begin{proposition}
  \label{ppo:emp}
  The empirical measure ${\bf p}_{emp}(t)$ \eqref{eq:empirical} satisfies:
  \begin{equation}
    \label{eq:empirical_SDE}
      \dd {\bf p}_{emp}(t) = \frac{1}{N}∑_{k=1,l=0}^{+∞} \Big(δ_{k-1} + δ_{l+1}- δ_{k}- δ_{l}\Big) \dd \mathrm{N}_t^{(k,l)}
    \end{equation}
    where $\mathrm{N}_t^{(k,l)}$ independent Poisson clock with intensity:
    \begin{equation}
      \label{eq:lambda_kl}
      λ_{k,l} = N⋅p_{emp,k}⋅(N⋅ p_{emp,l}-\mathbbm{1}_{\{k=l\}})⋅ \frac{λ}{k⋅N}
    \end{equation}
    where $p_{emp,k}$ is the $k-$th coordinate of ${\bf p}_{emp}$.
\end{proposition}
\begin{proof}
  Following the jump process given in \eqref{eq:jump_emp}, the empirical measure satisfies:
  \begin{equation}
    \dd {\bf p}_{emp}(t) = \frac{1}{N}∑_{i,j=1,i\neq j}^N \Big(δ_{S_i-1} + δ_{S_j+1}- δ_{S_i}- δ_{S_j}\Big) \dd \mathrm{N}_t^{(i,j)}
  \end{equation}
  Introducing $\mathrm{N}_t^{(k,l)}$ the Poisson process regrouping all the clocks corresponding to a giver with $k$ dollars giving to a receiver with $l$ dollars:
  \begin{equation}
    \label{eq:N_kl}
    \mathrm{N}_t^{(k,l)} = ∑_{\{i\neq j\,|\, S_i=k,S_j=l\}} \mathrm{N}_t^{(i,j)},
  \end{equation}
  In this sum, each clock $\mathrm{N}_t^{(i,j)}$ has the same intensity $λ/(S_i⋅N)=λ/(k⋅N)$. Moreover, counting the number of clocks involved in the sum~\eqref{eq:N_kl} leads to \eqref{eq:lambda_kl}. The indicator $\mathbbm{1}_{\{k=l\}}$ is here to remove the self-giving clocks $\mathrm{N}_t^{(i,i)}$: when an agent gives to itself, nothing happens.
\end{proof}

\begin{corollary}
  For any test function $φ$, the empirical measure ${\bf p}_{emp}(t)$ \eqref{eq:empirical} satisfies:
\begin{equation}
  \label{eq:empirical_SDE_weak}
    \dd 𝔼[⟨{\bf p}_{emp}(t),φ⟩] = λ𝔼[⟨Q_{rich}[{\bf p}_{emp}(t)],φ⟩]\dd t \;\; - \;\; \frac{λ}{N} 𝔼[⟨R[{\bf p}_{emp}(t)],φ⟩]\dd t
\end{equation}
where $Q_{rich}$ is the operator defined in \eqref{eq:Q_rich} and $R$ defined by:
\begin{equation}
  \label{eq:R}
  R[{\bf p}]_n:= \frac{p_{n+1}}{n+1} + \frac{p_{n-1}}{n-1}\mathbbm{1}_{\{n\geq 2\}} - \frac{2}{n}p_n\mathbbm{1}_{\{n\geq 1\}}.
\end{equation}
\end{corollary}
\begin{proof}
From the proposition \ref{ppo:emp}, we find:
\begin{eqnarray*}
  \dd 𝔼[⟨{\bf p}_{emp}(t),φ⟩] &=& 𝔼\left[∑_{k=1,l=0}^{+∞} \Big(φ(k-1) + φ(l+1)- φ(k)- φ(l)\Big) p_{emp,k}⋅ p_{emp,l}⋅ \frac{λ}{k}\right]\dd t \\
                              &&\quad - \frac{1}{N} 𝔼\left[∑_{k=1}^{+∞} \Big(φ(k-1) + φ(k+1)- 2φ(k)\Big) p_{emp,k}⋅ \frac{λ}{k}\right]\dd t\\
                              &=&  λ𝔼\left[∑_{k=1}^{+∞} \Big(φ(k-1)- φ(k)\Big) \frac{p_{emp,k}}{k}\right]\dd t \nonumber \\
  && \quad +   λ𝔼\left[∑_{l=0}^{+∞} \Big(φ(l+1)-φ(l)\Big) \overline{w}[{\bf p}_{emp}]⋅ p_{emp,l} \right]\dd t  \\
                              &&\quad - \frac{λ}{N} 𝔼\left[∑_{k=1}^{+∞} \Big(φ(k-1) + φ(k+1)- 2φ(k)\Big) p_{emp,k}⋅ \frac{1}{k}\right]\dd t
\end{eqnarray*}
where $\overline{w}[{\bf p}_{emp}]$ is defined in \eqref{eq:harmonic_average}. We recognize the weak formulation of $Q_{rich}$ \eqref{eq:Q_rich_weak} leading to \eqref{eq:empirical_SDE_weak}.
\end{proof}

The operator $R$~\eqref{eq:R} corresponds to the {\it bias} in the evolution of the empirical measure ${\bf p}_{emp}(t)$ compared to the evolution of ${\bf p}(t)$ solution to the limit equation \eqref{eq:law_limit_rich}. This bias vanishes as $λ/N$ goes to zero when the number of agents $N$ becomes large. The other source of discrepancy between ${\bf p}_{emp}(t)$ and ${\bf p}(t)$ is the {\it variance} of ${\bf p}_{emp}(t)$ (as it is a stochastic measure). Let's review an elementary result on compensated Poisson process.

\begin{remark}
  Denote $Z(t)$ a compound jump process and $M(t)$ its compensated version:
  \begin{equation}
    \label{eq:M_simple}
    \dd Z(t) = Y(t)\,\dd \mathrm{N}_t \quad,\quad M(t) = Z(t) - ∫_{0}^t μ(s)λ(s)\,\dd s
  \end{equation}
  where $Y(t)$ denotes the (independent) jumps and $\mathrm{N}_t$ Poisson process with intensity $λ(t)$ and $μ(t)=𝔼[Y(t)]$. The Ito's formula is given by:
  \begin{displaymath}
    \dd 𝔼[φ(M(t))] = 𝔼\Big[φ\Big(M(t-)+Y(t-)\Big)-φ(M(t-)\Big]λ(t)\dd t \;-\; 𝔼[φ'(M(t))μ(t)λ(t)]\,\dd t.
  \end{displaymath}
  In particular, for $φ(x)=x^2$, we obtain:
  \begin{eqnarray}
    \dd 𝔼[M^2(t)] &=& 𝔼[2M(t-)Y(t-)+Y^2(t-)]\,λ(t)\dd t - 𝔼[2M(t)μ(t)λ(t)]\,\dd t \nonumber \\
    &=& 𝔼[Y^2(t)]\,λ(t)\dd t. \label{eq:martingal_square}
  \end{eqnarray}
  Here, we assume that the jump $Y(t)$ is independent of the value $Z(t)$. To generalize the formula, one has to replace $μ(t)=𝔼[Y(t)]$ by $𝔼[Y(t)|Z(t)]$.
\end{remark}
Motivated by this remark, we obtain the following result.
\begin{proposition}
  Denote $M(t)$ the compensated process of the empirical measure ${\bf p}_{emp}(t)$:
\begin{equation}
  \label{eq:martingale}
  M(t) = {\bf p}_{emp}(t) - \left({\bf p}_{emp}(0) + λ∫_0^t\Big(Q_{rich}[{\bf p}_{emp}(s)]+\frac{1}{N}R[{\bf p}_{emp}(s)]\Big)\dd s  \right)
\end{equation}
then $M(t)$ is a $\ell^1$-value martingale and satisfies:
\begin{equation}
  \label{eq:bound_martingal}
  𝔼[\|M(t)\|_{\ell^1}] \leq \sqrt{\frac{4λ}{N}}\,t.
\end{equation}
\end{proposition}
\begin{proof}
  The key observation is that the jump \eqref{eq:jump_emp} for the empirical measure are of order $\mathcal{O}(1/N)$. Indeed:
  \begin{equation}
    E\left[\Big\|\frac{1}{N}(δ_{k-1} + δ_{l+1}- δ_{k}- δ_{l})\Big\|_{\ell^1}^2\right] \leq \frac{4}{N^2}.
  \end{equation}
  Applying the formula \eqref{eq:martingal_square} we obtain::
  \begin{equation}
    \dd 𝔼[\|M(t)\|_{\ell^1}^2] \leq ∑_{k=1,l=0}^{+∞} 𝔼\left[\frac{4}{N^2}\, ⋅\,Np_{emp,k}⋅ Np_{emp,l} \right] \frac{λ}{k⋅N}\dd t \;\leq\; \frac{4λ}{N} \dd t.
  \end{equation}
  Integrating in time gives \eqref{eq:bound_martingal}.
\end{proof}
We are now ready to prove the propagation of chaos for the rich-biased dynamics by showing that the empirical measure ${\bf p}_{emp}(t)$ converges to ${\bf p}(t)$ as $N→+∞$. The key
\begin{lemma}
  \label{lem:bound_Q_R}
  The operator $Q_{rich}$ \eqref{eq:Q_rich} is globally Lipschitz on $\ell^1(ℕ)\cap \mathcal{P}(ℕ)$ and $R$ is an bounded on $\ell^1(ℕ)$.
  \begin{eqnarray}
    \label{eq:lip_Q_rich}
    \|Q_{rich}[{\bf p}]-Q_{rich}[{\bf q}]\|_{\ell^1} &\leq& 4 \|{\bf p}-{\bf q}\|_{\ell^1} \qquad \text{for any } {\bf p},{\bf q}∈\ell^1(ℕ)\cap \mathcal{P}(ℕ)\\
    \label{eq_bound_R}
    \|R[{\bf p}]\|_{\ell^1} &\leq& 4 \|{\bf p}\|_{\ell^1} \qquad\qquad \text{for any } {\bf p}∈\ell^1(ℕ)
  \end{eqnarray}
\end{lemma}
\begin{proof}
  Since ${\bf p}∈\ell^1(ℕ)\cap \mathcal{P}(ℕ)$, the rate of receiving $w[{\bf p}]$ \eqref{eq:harmonic_mean} satisfies $0\leq w[{\bf p}]\leq 1$. Thus,
  \begin{displaymath}
    |Q_{rich}[{\bf p}]_n-Q_{rich}[{\bf q}]_n| \leq |p_{n+1}-q_{n+1}| + |p_{n-1}-q_{n-1}|+ 2 |p_{n}-q_{n}|.
  \end{displaymath}
  Summing in $n$ gives the result. We proceed similarly for the operator $R$.
\end{proof}

\begin{theorem}
  Consider ${\bf p}(t)$ solution to the limit equation \eqref{eq:law_limit_rich} and ${\bf p}_{emp}(t)$ empirical measure \eqref{eq:empirical}. Then:
  \begin{equation}
    \label{eq:diff_rho_emp_rho_limit}
    𝔼[\|{\bf p}_{emp}(t)-{\bf p}(t)\|_{\ell^1}] \leq \mathcal{O}\left(\frac{t \expo^{4λt}}{\sqrt{N}}\right),
  \end{equation}
  in particular ${\bf p}_{emp}(t) \stackrel{N → +∞}{\rightharpoonup} {\bf p}(t)$ for any $t\geq 0$.
\end{theorem}
\begin{proof}
  First we write down the integral form of the equation satisfied by both  ${\bf p}(t)$  and ${\bf p}_{emp}(t)$:
  \begin{eqnarray*}
    {\bf p}(t) &=& {\bf p}_0 + ∫_{0}^t Q_{rich}[{\bf p}(s)]\,\dd s \\
    {\bf p}_{emp}(t) &=& {\bf p}_0 + ∫_{0}^t Q_{rich}[{\bf p}_{emp}(s)]\,\dd s + \frac{1}{N}∫_{0}^t R[{\bf p}_{emp}(s)]\,\dd s + M(t)
  \end{eqnarray*}
  Combining the two equations give:
  \begin{eqnarray*}
    \|{\bf p}_{emp}(t)-{\bf p}(t)\|_{\ell^1} &\leq& λ∫_{0}^t \|Q_{rich}[{\bf p}_{emp}(s)]-Q_{rich}[{\bf p}(s)]\|_{\ell^1}\,\dd s  \\
                                             && \quad +\frac{λ}{N} ∫_{0}^t \|R[{\bf p}_{emp}(s)]\|_{\ell^1}\,\dd s + \|M(t)\|_{\ell^1}\\
    &\leq& 4λ∫_{0}^t \|{\bf p}_{emp}(s)-{\bf p}(s)\|_{\ell^1}\,\dd s +\frac{λ4t}{N} + \|M(t)\|_{\ell^1}
  \end{eqnarray*}
  using lemma \ref{lem:bound_Q_R}. Denoting $ϕ(t)=𝔼[\|{\bf p}_{emp}(t)-{\bf p}(t)\|_{\ell^1}]$, we deduce from the bound \eqref{eq:bound_martingal} of $M(t)$:
  \begin{displaymath}
    ϕ(t) \leq 4λ∫_{0}^t ϕ(s)\,\dd s + \frac{λ4t}{N} + \sqrt{\frac{4λ}{N}}t.
  \end{displaymath}
  Applying Gronwall's lemma leads to:
  \begin{displaymath}
    ϕ(t) \leq \left(\frac{λ4t}{N} + \sqrt{\frac{4λ}{N}}t\right) \expo^{4λt}
  \end{displaymath}
leading to the result.
\end{proof}

\begin{remark}
The martingale-based technique, developed in \cite{merle_cutoff_2019} and employed here for justifying the propagation of chaos, is remarkable since it does not require us to study the $N$-particle process $(S_1,\ldots, S_N)$ but solely its generator. One drawback is that this method might not work if the generator $Q$ of the limit process is unbounded, which is the case for the generator $Q_{poor}$ of the (limit) poor-biased dynamics \eqref{eq:law_limit_poor}.
\end{remark}

\subsection{Dispersive wave leading to vanishing wealth}
\label{sec:dispersive_wave}

As illustrated in the introduction (figure \ref{fig:illustration_3_dynamics_numerics}), the rich-biased dynamics tend to accentuate inequality, i.e. the Gini index $G(t)$ was approaching $1$ (its maximum value) for the agent-based model \eqref{rich_biased}~\eqref{eq:rich_bias}. We would like to investigate numerically the behavior of the solution to the rich-biased dynamics using the limit equation \eqref{eq:law_limit_rich} and the distribution ${\bf p}(t)=(p_0(t),p_1(t),\ldots)$.

\begin{figure}[h!]
  \centering
  \includegraphics[width=.97\textwidth]{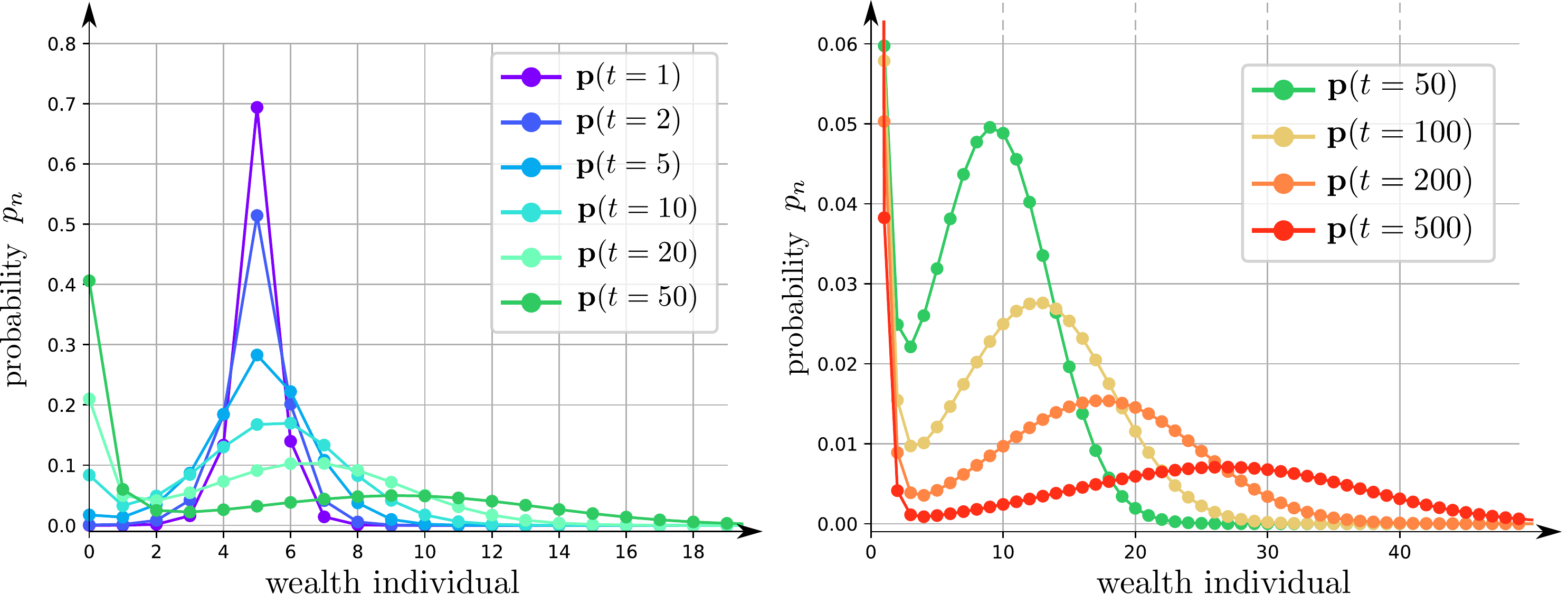}
  \caption{Evolution of the wealth distribution ${\bf p}(t)$ for the rich-biased dynamics \eqref{eq:law_limit_rich}. The distribution spreads in two parts: a large proportion starts to concentrate at zero (``poor distribution'') and while the other part form a dispersive traveling wave. Parameters: $Δt=5⋅10^{-3}$, ${\bf p}(t)\approx (p_0(t),p_1(t),…,p_{1,000}(t))$. A standard Runge-Kutta of order $4$ has been used to discretize the system.}
  \label{fig:vanishing_money_rho}
\end{figure}

\begin{figure}[h!]
  \centering
  \includegraphics[width=.97\textwidth]{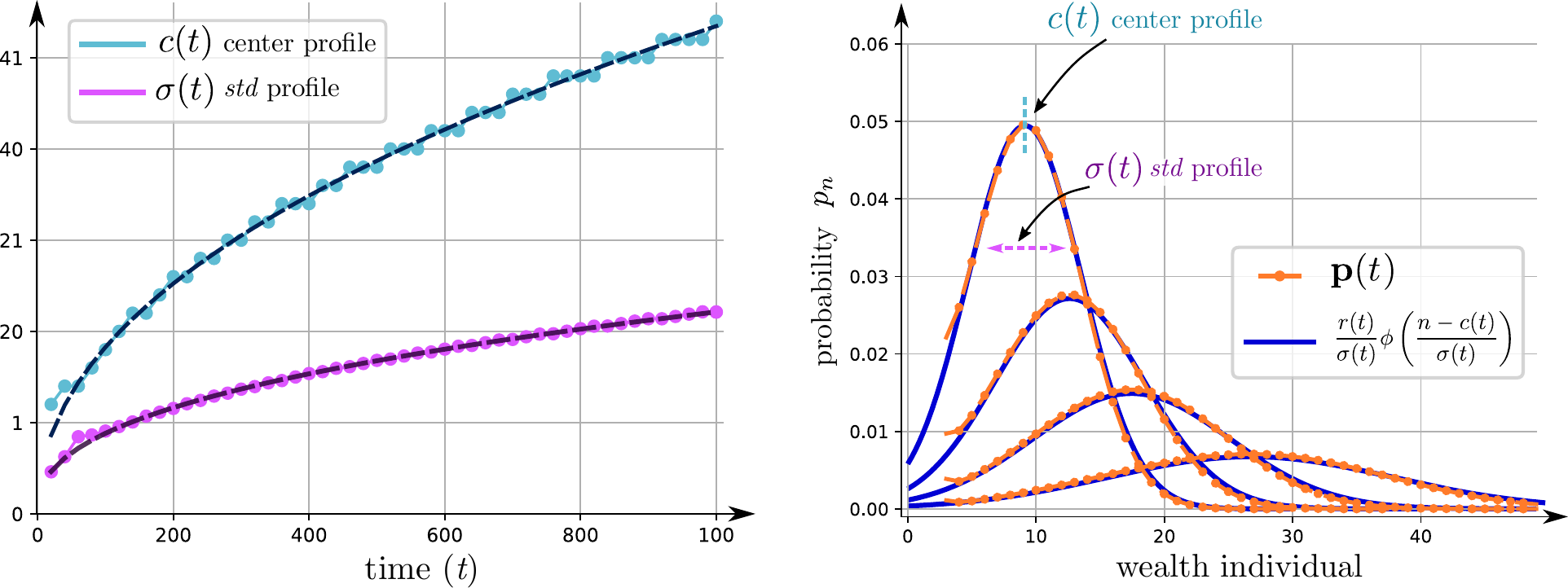}
  \caption{{\bf Left}: Estimation of the center $c(t)$ and standard deviation $σ(t)$ of the dispersive wave along with their parametric (power-law) estimation \eqref{eq:approx_c_sigma}. {\bf Right}: Comparison of the distribution ${\bf p}(t)$ (see figure \ref{fig:vanishing_money_rho}) with the dispersive wave using $ϕ$ the standard normal distribution.}
  \label{fig:fitting_profile}
\end{figure}

In figure \ref{fig:vanishing_money_rho}, we plot the evolution of the distribution ${\bf p}(t)$ starting from a Dirac distribution with mean $μ=5$ (i.e. $p_5=1$ and $p_i=0$ for $i\neq 5$). We observe that the distribution spreads in two parts: the bulk of the distribution moves toward zero whereas a smaller proportion is moving to the right. One can identify the two pieces as the ``poor'' and the ``rich''. Thus, the dynamics could be interpreted as the poor getting poorer and the rich getting richer. Notice that the proportion of poor is increasing (e.g. $p_0(t)$ is increasing) whereas the ``rich'' distribution resembles a dispersive traveling wave. Since both the total mass and the total amount of dollar are preserved (i.e. $∑_n n⋅p_n(t)=μ$ for any $t$), the dispersive traveling wave contains the bulk of the money but it is also vanishing in time.


To investigate more carefully the dispersive wave, we try to fit numerically its profile. After numerically examination, we choose to approximate by a Gaussian distribution. Meanwhile we approximate the ``poor'' distribution by a Dirac centered at zero $δ_0$. Thus, we approximate the distribution ${\bf p}(t)$ by the following Ansatz:
\begin{equation}
  \label{eq:profile}
  p_n(t)  \;\approx\; (1-r(t))⋅δ_0  \quad+\quad r(t)⋅\frac{1}{σ(t)}ϕ\left(\frac{n-c(t)}{σ(t)}\right),
\end{equation}
where $ϕ$ is the standard normal distribution (i.e. $ϕ(x)=\expo^{-x^2/2}/\sqrt{2π}$), $c(t)$ is the center of the profile, $σ(t)$ its standard deviation and $r(t)$ the proportion of rich. The speed of the wave $c(t)$ and its standard deviation $σ(t)$ are estimated numerically and plotted in figure \ref{fig:fitting_profile}. Their growth is well-approximated by a power-law of the form:
\begin{equation}
  \label{eq:approx_c_sigma}
  c(t) = 1.4748⋅t^{.466} \quad,\quad σ(t) = 0.9261⋅t^{.399}.
\end{equation}
Since the total amount of money is preserved, the proportion of rich $r(t)$ can be easily deduced from $c(t)$ since we must have $μ=r(t)⋅c(t)$. Such approximation leads to the fitting in figure \ref{fig:vanishing_money_rho}-right (dotted-black curves). We notice that the proportion of rich in our Ansatz is vanishing:
\begin{equation}
  \label{eq:rich_vanishing}
  r(t) = \frac{μ}{c(t)} \stackrel{t → +∞}{⟶} 0.
\end{equation}
Thus, we make the conjecture that ${\bf p}(t)$ converges weakly toward $δ_0$, i.e. all the money will asymptotically disappear.

To further assess our conjecture, we measure the evolution of the Gini index for the distribution ${\bf p}(t)$:
\begin{equation}
  \label{eq:gini_pdf}
  G[{\bf p}] = \frac{1}{2μ}∑_{i=0}^{+∞}∑_{j=0}^{+∞} |i-j| p_i p_j
\end{equation}
with $μ$ the standard mean. Using the Ansatz \eqref{eq:profile}, we can approximate the value of the Gini index given (see appendix \ref{formula_gini_dispersive}):
\begin{equation}
  \label{eq:gini_approx}
  G(t) \approx 1 - \frac{μ}{c(t)} + \frac{μ⋅σ(t)}{\sqrt{π}\,c^2(t)}.
\end{equation}
We plot in figure \ref{fig:vanishing_money_gini}-left the evolution of the Gini index $G(t)$ along with its approximation \eqref{eq:gini_approx}. We observe a good agreement between the two curves. To examine closely the long time behavior of the curves, we plot the evolution of $1-G(t)$ in log-scales (figure \ref{fig:vanishing_money_gini}-right) over a longer time interval (up to $t=10^5$). Both curves seem to converges similarly toward $0$ (indicating that $G(t) \stackrel{t → +∞}{⟶} 1$) with a slight overshoot for the Ansatz. This overshoot might be due to our approximation that the ``poor distribution'' of ${\bf p}(t)$ is concentrated exactly at zero (i.e. $(1-r(t))δ_0$). This approximation amplifies the inequality between the ``poor'' and ``rich'' parts of the distribution  and hence increases slightly  the Gini index. But overall the asymptotic behavior of the Gini index for ${\bf p}(t)$ matches with the formula \eqref{eq:gini_approx} and thus strengthens our assumption that ${\bf p}(t)$ will converge (weakly) to a Dirac $δ_0$. However, further analytically studies are needed to derive the asymptotic behavior of ${\bf p}(t)$ directly from the rich-biased evolution equation \eqref{eq:law_limit_rich}.

\begin{figure}[h!]
  \centering
  \includegraphics[width=.97\textwidth]{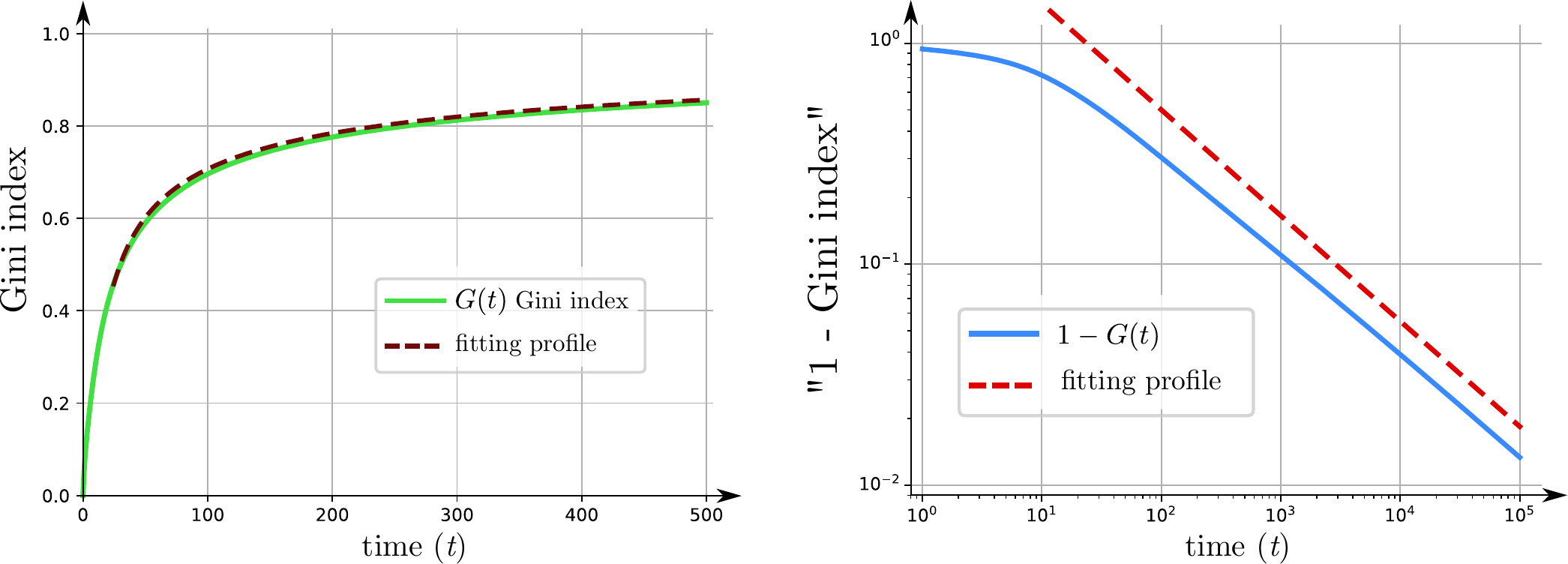}
  \caption{{\bf Left}: Evolution of the corresponding Gini index \eqref{eq:gini_pdf} along with the analytical approximation using the dispersive wave profile \eqref{eq:gini_approx}. {\bf Right} The Gini index converges to $1$ due to the vanishing dispersive wave transporting all the wealth to infinity.}
  \label{fig:vanishing_money_gini}
\end{figure}

\section{Conclusion}
\label{sec:conclusion}

In this manuscript, we have investigated three related models for money exchange originated from econophysics. For the unbiased and poor biased dynamics, we rigorously proved the so-called \emph{propagation of chaos} by virtue of a coupling technique, and we found an explicit rate of convergence of the limit dynamics for the poor biased model thanks to the Bakry-Emery approach. We have also introduced a more challenging dynamics referred to as the rich biased model, and a propagation of chaos result was established via a powerful martingale-based argument presented in \cite{merle_cutoff_2019}. In contrast to the two other dynamics, the rich-biased dynamics do not converge (strongly) to an equilibrium. Instead, we have found numerically evidence of the emergence of a (vanishing) dispersive wave. Such wave of extreme wealthy individual increases the inequality in the wealth distribution making the corresponding Gini index converging to its maximum $1$.

Although we have shown numerically strong evidence of a dispersive wave, it is desirable to derive such emerging behavior directly from the evolution equation. One direction of future work would be to derive space continuous dynamics of evolution equations in order to investigate analytically the profile of traveling waves. However, space continuous description such as the uniform reshuffling model could lead to additional challenges. For instance, proving propagation of chaos using the martingale technique for the uniform reshuffling model was more involved \cite{cao_entropy_2021}.

From a modeling perspective, one should explore how selecting the "receiver" as well as the "giver" could impact the dynamics. Indeed, in the three dynamics studied in the manuscript, the re-distribution process (how the one-dollar is redistributed) is uniform among all the agent. It would be reasonable to have the redistribution of the dollar based on the individual wealth (e.g. poor individual being more likely to receive a dollar). The interplay between receiver and giver selection could lead to novel emerging behaviors.

\section{Appendix}
\setcounter{equation}{0}

\subsection{Proof of lemma \ref{lem:tightness}}
\label{lemma:proof_tightness}
\begin{proof}
  Suppose first that the stochastic process $(S_1,…,S_N)$ satisfies the propagation of chaos. Let $φ$ be a test function,  $Z^{(N)} = ⟨ρ_{emp}^{(N)},φ⟩$ a random variable and $\overline{Z} = \mathbb{E}[φ(\overline{S}_1)]$ a constant. For notation convenience, we write $[N] := \{1,2,\ldots,N\}$. To prove that $Z^{(N)}$ converges in law to $\overline{Z}$, it is sufficient to prove the convergence in $L^2$:
  \begin{eqnarray}
    \mathbb{E}[|Z^{(N)}-\overline{Z}]|^2] &=& \mathbb{E}\left[\left|\frac{1}{N}∑_{i=1}^N φ(S_i) - \overline{Z}\right|^2\right] \\
                                          &=& \frac{1}{N^2}∑_{i,j,i\neq j} \mathbb{E}[φ(S_i)φ(S_j)]  + \frac{1}{N^2}∑_{i=1}^N \mathbb{E}[φ^2(S_i)] \nonumber \\
                                          && \qquad - \frac{2}{N}∑_{i=1}^N \mathbb{E}[φ(S_i)]\overline{Z} \;\;   +  \;\; \overline{Z}^2 \nonumber \\
                                          &\stackrel{N → +∞}{⟶} & \overline{Z}^2 + 0 - 2 \overline{Z}⋅\overline{Z} + \overline{Z}^2 = 0 \nonumber
  \end{eqnarray}
  using \eqref{eq:process} with $k=2$ and $k=1$.

  Proving the converse is more challenging. Let's take as test function $φ(s_1,…,s_k)=φ_1(s_1)…φ_k(s_k)$ and denote the random variable $Z_i = ⟨ρ_{emp}^{(N)},φ_i⟩$ for all $i$. By assumption, $Z_i$ converges in law to the constant $⟨\overbar{ρ}_1,φ_i⟩=\mathbb{E}[φ_i(\overline{S}_1)]$. We deduce:
  \begin{eqnarray*}
    \Big|\mathbb{E}[φ_1(S_1) \cdots φ_k(S_k)] -\mathbb{E}[φ_1(\overline{S}_1)\cdots φ_k(\overline{S}_k)]\Big| &\leq& \Big|\mathbb{E}[φ_1(S_1)\cdots φ_k(S_k)] - \mathbb{E}[Z_1\cdots Z_k]\Big|  \\
                                                                                                              && \qquad  + \Big|\mathbb{E}[Z_1\cdots Z_k]  - \mathbb{E}[φ_1(\overline{S}_1)\cdots φ_k(\overline{S}_k)]\Big|\\
                                                                                                              &=:& |A|\;+\;|B|.
  \end{eqnarray*}
  Since each $Z_i$ converges to the constant $\mathbb{E}[φ_i(\overline{S}_i)]$, all the product in $B$ convergence to zero using Slutsky's theorem. For $A$, we use the invariance by permutations:
  \begin{eqnarray}
    A &=& \frac{1}{N!} ∑_{σ ∈ \mathcal{S}_N} \mathbb{E}[φ_1(S_{σ(1)}) \cdots φ_k(S_{σ(k)})] - \frac{1}{N^k}\!\!\!\! ∑_{(i_1,…,i_k)∈[N]^k}\!\!\!\!\!\!\!\mathbb{E}[φ_{1}(S_{i_1}) \cdots φ_{k}(S_{i_k})] \\
      &=& \frac{(N-k)!}{N!} ∑_{(i_1,…,i_k)∈\mathcal{P}_{N,k}}\!\!\!\!\!\!\!\!\mathbb{E}[φ_{1}(S_{i_1}) \cdots φ_{k}(S_{i_k})] - \frac{1}{N^k} ∑_{(i_1,…,i_k)∈[N]^k}\!\!\!\!\!\!\!\!\mathbb{E}[φ_{1}(S_{i_1}) \cdots φ_{k}(S_{i_k})], \nonumber
  \end{eqnarray}
  where $\mathcal{P}_{N,k}⊂[N]^k$ is the set of all the permutations of $k$ elements in $[N]$ and in particular $|\mathcal{P}_{N,k}|=N!/(N-k)!$. To conclude, we split the set $[N]^k$ in two parts:
  \begin{eqnarray*}
    A &=& \left(1-\frac{1}{N^k}⋅\frac{N!}{(N-k)!}\right)⋅\frac{(N-k)!}{N!} ∑_{(i_1,…,i_k)∈\mathcal{P}_{N,k}}\!\!\!\!\!\!\!\!\mathbb{E}[φ_{i_1}(S_{i_1}) \cdots φ_{i_k}(S_{i_k})] \\
      && \quad - \frac{1}{N^k} ∑_{(i_1,…,i_k)∈[N]^k \setminus \mathcal{P}_{N,k}}\!\!\!\!\!\!\!\!\mathbb{E}[φ_{1}(S_{i_1}) \cdots φ_{k}(S_{i_k})]
  \end{eqnarray*}
  Thus, denoting $C$ an upper-bound for any $\mathbb{E}[φ_{1}(S_{i_1}) \cdots φ_{k}(S_{i_k})]$:
  \begin{eqnarray}
    |A| &\leq& \left(1-\frac{N!}{N^k(N-k)!}\right) C \;\;+\;\; \frac{1}{N^k}(N^k - |\mathcal{P}_{N,k}|) C \\
        &=& 2\left(1-\frac{N!}{N^k(N-k)!}\right) C
            = 2\left(1 - \frac{(N\!-\!k+1)}{N}\cdots \frac{N\!-\!1}{N}⋅\frac{N}{N}\right)C \;\;\stackrel{N → +∞}{⟶} 0. \nonumber
  \end{eqnarray}
\end{proof}

\subsection{Gini index dispersive wave}
\label{formula_gini_dispersive}

We estimate the Gini coefficient for a (continuous) distribution of the form:
\begin{equation}
  \label{eq:asymp_rho}
  ρ(x) = (1-r)⋅δ_0(x)  + r⋅\frac{1}{σ}ϕ\left(\frac{x-c}{σ}\right)
\end{equation}
where $ϕ$ is the standard normal distribution, $r,c,σ$ some positive constant with $r∈[0,1]$. The law $ρ$ can be represented by a random variable:
\begin{equation}
  \label{eq:asymp_X}
  X = (1-Y)⋅0 + Y⋅(c+σZ)
\end{equation}
with $Y$ random Bernoulli variable with probability $r$ (i.e. $Y \sim B(r)$), $Z$ a random variable with normal law (i.e. $Z\sim \mathcal{N}(0,1)$), $Y$ and $Z$ being independent. To estimate the Gini index of $ρ$, we take two independent random variables $X_1$ and $X_2$ with such law and estimate the expectation of their difference:
\begin{eqnarray}
  G &=& \frac{1}{2μ}𝔼[|X_1-X_2|] = \frac{1}{2μ}𝔼[|Y_1⋅(c+σZ_1) \;-\; Y_2⋅(c+σZ_2)|] \nonumber \\
  &=& \frac{1}{2μ}𝔼[|c(Y_1-Y_2) + σ(Y_1Z_1-Y_2Z_2)|]
\end{eqnarray}
We then take the conditional expectation with respect to $Y_1$ and $Y_2$:
\begin{eqnarray}
  2μG &=& 0+ 𝔼[|c + σZ_1|]ℙ[Y_1=1,Y_2=0] \nonumber \\
      && \quad + 𝔼[|-c - σZ_2|]ℙ[Y_1=0,Y_2=1] \nonumber \\
  && \qquad + 𝔼[| σ(Z_1-Z_2)|]ℙ[Y_1=1,Y_2=1] \nonumber \\
    &=& 2⋅𝔼[|c + σZ_1|]r(1-r) + 𝔼[| σ(Z_1-Z_2)|]r^2
\end{eqnarray}
For large $c$, we made the approximation $𝔼[|c + σZ_1|]\approx 𝔼[c + σZ_1] = c$. Moreover, the expectation of the difference between two standard Gaussian random variables is known explicitly: $𝔼[| Z_1-Z_2|] = 2/\sqrt{π}$. We deduce:
\begin{equation}
  2μG \approx 2c⋅r(1-r) + σ \frac{2}{\sqrt{π}} r^2.
\end{equation}
Furthermore, if $r=μ/c$, we obtain:
\begin{equation}
  G \approx 1 - \frac{μ}{c} + \frac{σ μ}{\sqrt{π} c^2}.
\end{equation}


\bibliographystyle{plain}
\bibliography{propa_chaos_econophysics}

\end{document}